\documentclass[oneside]{scrartcl}

\setkomafont{sectioning}{\bfseries}
\usepackage[utf8]{inputenc}\usepackage{lmodern}
\usepackage[T1]{fontenc}
\usepackage{verbatim}
\usepackage{amssymb}\usepackage[sans]{dsfont}
\usepackage{amsbsy,bbm,fontawesome,nicefrac,tikzsymbols,hyperref,bclogo,multirow,yhmath,bbding}
\usepackage{soul}

\usepackage{amsmath,amssymb,amsthm,mathrsfs, cancel}
\usepackage{natbib}

\makeatletter
\theoremstyle{plain}

\newtheorem{theorem}{Theorem}
\newtheorem{proposition}[theorem]{Proposition}                 
\newtheorem{lemma}[theorem]{Lemma}
\newtheorem{corollary}[theorem]{Corollary}

\theoremstyle{definition}
\newtheorem{definition}[theorem]{Definition}\newtheorem{annahme}{Assumption}                                       
\theoremstyle{remark}
\newtheorem{remark}[theorem]{Remark}

\def\ps{\Pi}
\def\co{\mathbbm{L}}

\newcommand{\lebesgue}{\boldsymbol{\lambda}}
\def\C{\mathbbmtt{C}}

\def\E{\mathbb{E}}
\def\G{\mathbbm{G}}
\def\H{\mathbbm{H}}
\def\F{\mathscr{F}}
\def\bdg{\operatorname{c}}

\def\mom{\mathbbmtt{C}_{\operatorname{mo}}}

\def\N{\mathbb{N}}
\def\p{\mathscr{P}}
\def\R{\mathbb{R}}

\def\rd{\mathbbm{R}}
\def\V{\mathbb{V}}

\def\Ma{\mathbbm{M}}
\def\M{\mathds{M}}
\def\ML{\mathcal{L}}

\def\SS{\mathcal{S}}

\def\N{\mathbb{N}}
\def\P{\mathbb{P}}

\def\prox{\operatorname{prox}}
\def\supp{\operatorname{supp}}

\def\h{{\mathbf{h}}}

\def\X{X}

\renewcommand{\d}{\mathrm{d}}
\def\L{L}
\renewcommand{\a}{^{\frac{1}{\alpha}}}
\renewcommand{\p}{^{\frac{1}{p}}}
\newcommand{\e}{\mathrm{e}}
\newcommand{\U}{U}

\newcommand{\Var}{\operatorname{Var}}

\renewcommand{\AA}{\mathds{A}}

\newcommand{\EE}{\mathcal{E}}
\newcommand{\FF}{\mathcal{F}}
\newcommand{\GG}{\mathcal{G}}

\newcommand{\LL}{\mathcal{L}}
\newcommand{\ep}{\varepsilon}

\newcommand{\Ph}{\Phi^{b}}
\newcommand{\Ps}{\Pi^{b}}

\renewcommand{\c}{^{\operatorname{c}}}
\newcommand{\VV}{\mathcal{V}}

\renewcommand{\Upsilon}{Z}
\renewcommand{\tilde}{\widetilde}%

\usepackage[left=2.7cm,right=2.7cm,top=1.5cm,bottom=1.5cm,includeheadfoot]{geometry}
\title{Concentration of scalar ergodic diffusions and some statistical implications
}

\author{Cathrine Aeckerle-Willems and Claudia Strauch\thanks{Universit\"at Mannheim, Institut f\"ur Mathematik, 68131 Mannheim, Germany.\newline
\noindent E-mail: aeckerle@uni-mannheim.de/strauch@uni-mannheim.de}}
\date{\vspace*{-1.5em}}
\begin{document}

\maketitle

\begin{abstract}
We derive uniform concentration inequalities for continuous-time analogues of empirical processes and related stochastic integrals of scalar ergodic diffusion processes. 
Thereby, we lay the foundation typically required for the study of $\sup$-norm properties of estimation procedures for a large class of diffusion processes. 
In the classical i.i.d.~context, a key device for the statistical $\sup$-norm analysis is provided by Talagrand-type concentration inequalities. 
Aiming for a parallel substitute in the diffusion framework, we present a systematic, self-contained approach to such uniform concentration inequalities via martingale approximation and moment bounds obtained by the generic chaining method. 
The developed machinery is of independent probabilistic interest and can serve as a starting point for investigations of other processes such as more general Markov processes, in particular multivariate or discretely observed diffusions. 
As a first concrete statistical application, we analyse the $\sup$-norm error of estimating the invariant density of an ergodic diffusion via the natural local time estimator and the classical nonparametric kernel density estimator, respectively.

\end{abstract}

\section{Introduction}\label{sec:intro}
With regard to the very basic idea of estimating expected values via sample means as motivated by the law of large numbers, the relevance of concentration inequalities which quantify the deviation behaviour of more general additive functionals from their mean is pretty obvious. 
It is thus natural that they can be identified as being a central device in many statistical investigations, both from a frequentist and a Bayesian point of view.
From an applied perspective, expected maximal errors describing worst case scenarios are of particular interest for quantifying the quality of estimators. 
The analysis of $\sup$-norm risk criteria when estimating densities, regression functions or other characteristics thus is of immense relevance.
Nevertheless, even in classical situations like density estimation from i.i.d.~observations, the $\sup$-norm case is a delicate issue and usually not treated as exhaustively as $L^p$ or pointwise risk measures. 
Analysing the $\sup$-norm risk often requires to resort to empirical process theory.
More precisely, it leads to the need of finding moment bounds and concentration inequalities for the supremum of empirical processes, i.e., the supremum of additive functionals, over possibly infinite-dimensional function classes. 
This turns out to be a probabilistic challenge. 
In case of diffusion processes with unbounded state space, estimation of diffusion characteristics in $\sup$-norm risk is a mostly open question even in the most basic setting of continuous observation of a scalar process.
The current work aims at providing the fundamental probabilistic tool box, including uniform concentration inequalities for empirical processes and related concepts, in the continuous scalar diffusion context as they are essential for further statistical research on the $\sup$-norm risk. 

Since they are taken as a standard model for a number of random phenomena arising in various applications, statistical inference for ergodic diffusion processes, based on different observation schemes, has been widely developed during the past decades. 
While observation data as the central ingredient of any estimation procedure in practice are always discrete, it is insightful to start the statistical analysis in the framework of continuous observations, thereby providing both benchmark results and a starting point for estimation schemes based on discrete data. 
Within this framework, we demonstrate that our approach to concentration results can be specified as needed for proving sharp upper bounds on $\sup$-norm risks. 
Moreover, we introduce a machinery for obtaining uniform concentration inequalities for empirical processes based on martingale approximation and the generic chaining device that allows for the analogue treatment of more general classes of Markov processes as well.  
In particular, with regard to the diffusion process set-up, our approach could also be adapted for $\sup$-norm risk investigations based on discrete observations or multivariate state variables.
While the basic idea of martingale approximation is applied at several places in the statistical literature, we are not aware of \emph{any} systematic attempts to exploit the approach for deriving concentration results.

\subsection{Basic framework and main results}
Given a continuous-time Markov process $\X=(X_t)_{t\geq 0}$ with invariant measure $\mu$, the counterpart to the empirical process 
$\sqrt n\left(n^{-1}\sum_{i=0}^{n-1} f(Y_i) - \E\left[f(Y_0)\right]\right)$, $f\in\mathcal F$,
based on i.i.d.~observations $Y_0,..., Y_{n-1}$, is given as 
\begin{equation}\label{emp_pro}
\sqrt t\left(\frac{1}{t}\int_0^t f(X_s)\d s - \int f(x)\d\mu (x)\right),\quad f\in\mathcal F,
\end{equation} 
$\mathcal F$ denoting a class of functions, typically satisfying suitable entropy conditions. 
This con\-ti\-nu\-ous-time version of the classical empirical process is our first object of interest.
For the goal at hand, we will focus on diffusion processes given as a solution of the SDE
\begin{equation}\label{intro}
\d X_t\ =\ b(X_t)\d t+ \sigma(X_t)\d W_t, \quad X_0=\xi,\ t\geq 0,
\end{equation}
where $W$ is a standard Brownian motion and the initial value $\xi$ is a random variable independent of $W$. 
We restrict to the ergodic case where the Markov process $X$ admits an invariant measure, and we denote by $\rho_b$ and $\mu_b$ the invariant density and the associated invariant measure, respectively. 
Furthermore, we will always consider stationary solutions of \eqref{intro}, i.e., we assume that $\xi\sim\mu_b$.
In this framework, we will also provide precise uniform concentration inequalities for stochastic integrals
\begin{equation}\label{emp_stoch_pro}
\frac{1}{t}\int_0^t f(X_s)\d X_s - \E\left[f(X_0)b(X_0)\right], \quad f\in\mathcal F,
\end{equation} 
which turn out to be essential for statistical investigations.

\paragraph{Main results}
For a diffusion process given as the stationary solution of \eqref{intro}, Theorem \ref{theo:bernsup} provides an exponential tail inequality for
\[
\sup_{f\in\mathcal F}\ \sqrt t\ \Big| \frac{1}{t}\int_0^t f(X_s)\d s - \E\left[f(X_0)\right]\Big|\] 
as well as bounds on its $p$-th moments, for any $p\geq 1$, under standard entropy conditions on the function class $\mathcal F$.
Proposition \ref{theo:mart} and Theorem \ref{improved_version} constitute analogue results for the supremum of the stochastic integrals \eqref{emp_stoch_pro}. 
We emphasise at this point that we allow for unbounded functions $f\in\mathcal F$ which is even for nonuniform Bernstein-type results absolutely nonstandard. 
Furthermore, we introduce a localisation procedure which allows to look at processes on the whole real line instead of compacts. 

As a statistical application, we investigate nonparametric invariant density estimation in su\-pre\-mum-norm based on a continuous record of observations $(X_t)_{0\leq t \leq T}$ of the solution of \eqref{intro} started in the equilibrium. 
In the continuous framework, the local time -- which can be interpreted as the derivative of an empirical distribution function -- naturally qualifies as an estimator of this density. 
Corresponding upper bounds for all $p$-th moments of the $\sup$-norm loss are given in Corollary \ref{cor:centlt}. 
We advocate the investigation of the continuous, scalar case because it serves as a fundament and as a relevant benchmark for further investigations of discrete observation schemes and the multivariate case. 
With this purpose in mind, the density estimator based on local time is not the preferable choice as it does not open immediate access to discrete-time or multivariate estimators. 
In contrast, the very classical kernel (invariant) density estimator meets all these requirements, and it achieves the same (optimal) $\sup$-norm rates of convergence which we establish in Corollary \ref{prop:csi}.  

\subsection{Structure and techniques: an overview}
\paragraph{Introducing methods at the concrete example of a tail estimate for the local time}
We will start in Section \ref{sec:2} with an exponential uniform upper tail inequality for the local time of a continuous semimartingale, stated in Theorem \ref{theo:lt}. 
The local time of semimartingales was discussed by \cite{MR0501332}, and we adopt his definition:
Given a continuous semimartingale $\X$, denote by $(\L_t^a(\X))_{t\geq 0}$, $a\in\R$, the \emph{local time of $X$ at level $a$}, i.e., the increasing process which satisfies the following identity,
\begin{equation}\label{eq:tanaka}
(X_t-a)^-\ =\ (X_0-a)^--\int_0^t\mathds{1}\{X_s\leq a\}\d X_s+\frac{1}{2}\ L_t^a(X), \quad t>0,\ a\in\R.
\end{equation}
We have chosen to begin from this special case not only because of the statistical interest in the local time. 
It is instructive since, in the process of proving Theorem \ref{theo:lt}, we will already introduce key ideas and methods, including the generic chaining and localisation procedures that we will resort to for the further analysis of general empirical processes. 
From the representation \eqref{eq:tanaka} it actually becomes clear that analysing $\sup_{a\in\R} L_t^a(X)$ requires looking at 
\[\sup_{f\in\mathcal F}\Big|\int_0^t f(X_s)\d X_s\Big|,\quad\text{ for }\mathcal F:=\left\{\mathds{1}\{\ \cdot\ \leq a\}\ \colon\ a\in\R\right\}.\] 
This expression accounts for the connection to the investigation of uniform concentration inequalities for empirical processes and stochastic integrals as in \eqref{emp_stoch_pro}.
The proof thus serves as a blueprint and a concrete example that prevents from losing track while handling the technicalities coming up in the general empirical process setting. 
Under suitable moment conditions, we do not have to restrict to diffusion processes, yet. 
Instead, the results presented in Section \ref{sec:2} hold in a general continuous semimartingale framework.

A central ingredient of the proof of Theorem \ref{theo:lt} is the decomposition of the local time into a martingale part and a remaining term induced by \eqref{eq:tanaka}. 
Considering more general additive functionals as in \eqref{emp_pro}, we carry on this idea and prove a uniform concentration inequality for empirical processes \eqref{emp_pro} of general continuous semimartingales, assuming the existence of a martingale approximation.

\paragraph{Martingale approximation}
In the discrete framework, the technique of martingale approximation was initiated by \cite{MR0501277}, while \cite{bhat82} proved the continuous-time analogue.
Their basic idea consists in deriving the CLT for processes $\G_t(f)$,
\[
\G_t(f)\ := \ \sqrt t\left(\frac{1}{t}\int_0^t f(X_s)\d s-\E[f(X_0)]\right),
\]
$f$ some square-integrable function, by decomposing the above partial sums into the sum of a martingale with stationary increments and a remainder term. 
Asymptotic normality then follows from a martingale CLT.
For fixing terminology, suppose that $\G_t(f)$, $f\colon\R\to\R$, lives on a fixed filtered probability space $(\Omega,\mathscr{F},(\F_t)_{t\geq0},\P)$.
One then says that there exists a \emph{martingale appro\-xi\-ma\-tion to $\G_t(f)$}, $f\colon\R\to\R$, if there exist two processes $(M_t(f))_{t\geq0}$ and $(R_t(f))_{t\geq0}$ on $(\Omega,\mathscr{F},(\F_t)_{t\geq0},\P)$ such that
\begin{equation}\label{eq:martapp}
\G_t(f)\ =\ \frac{1}{\sqrt t}\ M_t(f)+\frac{1}{\sqrt t}\ R_t(f),\quad t>0,
\end{equation}
where $(M_t(f))_{t\geq0}$ is a martingale wrt $(\mathscr{F}_t)_{t\geq0}$ fulfilling $M_0(f)=0$ and the remainder term $(R_t(f))_{t\geq0}$ is negligible in some sense.

\paragraph{Results on uniform concentration for empirical processes of continuous semimartingales}
Given the availability of a suitable martingale approximation of the additive functional, we show in Section \ref{sec:3} how to derive uniform concentration results on $t^{-1}\int_0^t f(X_s)\d s$, $f\in\mathcal F$, in the continuous semimartingale setting. 
Speaking of uniform concentration results, we refer to inequalities of the form
\begin{equation}\label{conineqform}
\P\left(\sup_{f\in\mathcal F}\left|\G_t(f)\right|\geq \e\Phi(u)\right)\ \leq\ \exp(-u),\quad \text{for any } u\geq 1, \end{equation}
$\e$ denoting Euler's number,
which is an immediate consequence of the moment bound \[\left(\E\left[\sup_{f\in\mathcal F}\left|\G_t(f)\right|\right]^p\right)\p\ \leq\ \Phi(p)\] for some function $\Phi\colon (0,\infty)\to(0,\infty)$ and any $p\geq 1$.
Note that this is not a concentration inequality for the random variable $\sup_{f\in\mathcal F}t^{-1/2}\int_0^t f(X_s)\d s$ as such. 
It is rather a uniform or worst case statement on the concentration of $t^{-1/2}\int_0^t f(X_s)\d s$. 
Nonetheless, it additionally implies an upper exponential deviation inequality for the random variable \[\sup_{f\in\mathcal F}\frac{1}{\sqrt t}\Big|\int_0^t f(X_s)\d s\Big|\] from its mean.
These uniform concentration inequalities given in Theorem \ref{theo:1} are the main result in Section \ref{sec:3}. 
The tail behaviour incorporated in the nature of the function $\Phi$ in \eqref{conineqform} is described in terms of entropy integrals. 
This formulation is not the most handy but means a higher degree of generality. 
Of course, the entropy integrals can further be upper bounded under mild entropy conditions on the function class as known from the i.i.d.~set-up (see Lemma \ref{EntropyIntegralLebesgueMetric} of the Appendix).
The proof of Theorem \ref{theo:1} relies on a localised generic chaining procedure that can be applied assuming the existence of a martingale approximation of the empirical process $(\G_t(f))_{f\in\mathcal F}$.
Let us already note that our results on the concentration of empirical processes of the form \eqref{emp_pro} in Section \ref{sec:3} do \emph{not} require the existence of a local time process. 
Though the framework of continuous semimartingales is suitable for our goal of considering diffusion processes, the techniques could also be applied to other models, e.g., more general classes of Markov processes.
The only prerequisites consist in a maximal inequality of the form \eqref{maxineqX} and a martingale approximation with suitable moment bounds as in \eqref{martrep2}. 
We also advocate our approach as a starting point for the derivation of parallel results for multivariate diffusion processes.

\paragraph{Results on uniform concentration for empirical processes and stochastic integrals of scalar ergodic diffusions} 
The findings of Section \ref{sec:2} and Section \ref{sec:3} are applied to obtain uniform concentration results for $t^{-1}\int_0^t f(X_s)\d s$ and $t^{-1}\int_0^t f(X_s)\d X_s,\,f\in\mathcal F,$ in the diffusion framework in Section \ref{sec:cmei}.
For the concrete case of diffusion processes, we show in Section \ref{sec:4.2} that a suitable martingale approximation as described above exists.
This fact immediately implies the uniform concentration inequalities for empirical processes stated in Theorem \ref{theo:bernsup}.
The natural approach of analysing the supremum of these objects by exploiting concentration results such as Bernstein-type deviation inequalities for additive diffusion functionals has severe obstacles which are detailed in Remark \ref{rem:bern}.
In particular, this approach forces one to impose additional conditions on the characteristics of the diffusion process in order to prove the required \emph{uniform} concentration results.
Remarkably, the alternative strategy via martingale approximation allows to work under minimal assumptions on the class of diffusion processes.
As a consequence, we obtain results on the uniform concentration both of additive functionals and of stochastic integrals.

The uniform concentration inequality for the stochastic integrals of a diffusion process is subject of Proposition \ref{theo:mart} and makes use of Theorem \ref{theo:lt} on the local time. 
In Proposition \ref{theo:mart}, we consider the question of exploring the tail behaviour for quantities of the form
\[
\H_t(f)\ :=\ \sqrt t\left(\frac{1}{t}\int_0^t f(X_s)\d X_s-\E\left[f(X_0)b(X_0)\right]\right), \quad f\in\FF,
\]
$X$ some diffusion process solving \eqref{intro} and $\FF$ denoting some (possibly infinite-dimensional) class of integrable functions.
For adaptive procedures for estimating the characteristics of $\X$, one generally requires both an upper bound on 
\[\E\left[\sup_{f\in\FF}|\H_t(f)|\right],\quad\text{ $\FF$ some class of translated kernel functions,}\] and an upper tail bound for the deviation of the supremum. Using generic chaining methods initiated by Talagrand (cf.~\cite{tala14}), both can be derived by obtaining upper bounds for \emph{all} $p$-th ($p\geq1$) moments of $(\H_t(f))_{f\in\FF}$.

\paragraph{Uniform moment bounds and exponential inequalities for stochastic integrals via generic chaining}
Starting from the basic decomposition 
\begin{eqnarray}\nonumber
\H_t(f)&=& \frac{1}{\sqrt t}\int_0^t\left(f(X_s)b(X_s)-\E\left[f(X_0)b(X_0)\right]\right)\d s \ +\ \frac{1}{\sqrt t}\int_0^t f(X_s)\sigma(X_s)\d W_s\\\label{def:A}
&=:& (\mathbf{I})+(\mathbf{II}),
\end{eqnarray}
we recognise the empirical process $(\mathbf{I})$ which can be treated by means of Theorem \ref{theo:bernsup}. 
The next step then consists in finding upper bounds on the $p$-th moments of $(\mathbf{II})$.
Applying the Burkholder--Davis--Gundy (BDG) inequality and the occupation times formula, one obtains
\begin{align*}
\E\left[\left|\frac{1}{\sqrt t}\int_0^tf(X_s)\sigma(X_s)\d W_s\right|^p\right]
&\leq\ C_p\E\left[\left(\frac{1}{t}\int_0^tf^2(X_s)\sigma^2(X_s)\d s\right)^{p/2}\right]\\
&=\ C_p\E\left[\left(\frac{1}{t}\int_\R f^2(y)L_t^y(\X)\d y\right)^{p/2}\right]\\
&\leq\ C_pt^{-p/2}\left(\int_\R f^2(y)\d y\right)^{p/2}\ \E\left[\left(\sup_{a\in\R}|L_t^a(\X)|\right)^{p/2}\right].
\end{align*}
At first sight, this upper bound may seem to be very rough, but looking into the details of the proof, it becomes clear that one needs to obtain the $L^2$ norm of $f$ on the right hand side for the generic chaining procedure which accounts for this estimate. 
Conveniently, we can then apply Theorem \ref{theo:lt}.
It provides both an upper bound on the $p$-th moments $\E\left[\left\|\L_t^\bullet(X)\right\|_\infty^p\right]$ and a corresponding tail estimate.
Inspection of the proof of Theorem \ref{theo:lt} shows that it relies on three substantial ingredients:
\begin{itemize}
\item[(i)]
The proof exploits the decomposition of the local time process into a martingale part and a remainder term provided by Tanaka's formula. The analysis of the martingale part then relies on generic chaining methods.
\item[(ii)]
The latter requires the increments of the martingale to exhibit a subexponential tail behaviour wrt to a suitable metric (cf.~\eqref{ineq:a}). 
We discover this relation from a \emph{sharp} formulation of the bound
\[
\E\left[\left(\int_0^t\mathds{1}\{a\leq X_s\leq b \}\d \langle M,M\rangle_s\right)^p\right]
\ \leq\ c_p(b-a)^p\bigg\{\E\left[\langle M,M\rangle_t^{\frac{p}{2}}\right]+\E\left[\left(\int_0^t|\d V_s|\right)^p\right]\bigg\}
\] 
(see, e.g., Lemma 9.5 in \cite{legall16}), for $M=(M_t)_{t\geq 0}$ and $V=(V_t)_{t\geq 0}$ denoting the martingale part and the finite variation part of the semimartingale $X$, respectively. 
Here, `sharp' refers to the dependence of the constant $c_p$ on the order $p$ of the moments.
This can be obtained by means of Proposition 4.2 in \cite{bayo82} (see \eqref{4.2} below).
\item[(iii)]
The supremum taken over the entire real line is dealt with by an investigation of the random, compact support of the local time $\L_t^\bullet(\X)$.
In particular, we rely on a maximal inequality for the process $X$ which allows to control the probability that the support of the process exceeds certain levels.
\end{itemize}
As already announced, the proof of Theorem \ref{theo:lt} also serves as a blueprint for the analysis of the supremum of additive functionals in Theorem \ref{theo:1}.

There is some evidence of the statistical relevance of diffusion local time. 
As one first concrete example, let us mention the deep Donsker-type theorems for diffusion processes in \cite{vdvvz05} whose proof relies on a limit theorem for the supremum of diffusion local time.
Another instance concerns the completely different context of studying nonparametric Bayesian procedures for one-dimensional SDEs:
\cite{poketal13} investigate a Bayesian approach to nonparametric estimation of the periodic drift of a scalar diffusion from continuous observations and derive bounds on the rate at which the posterior contracts around the true drift in $L^2$ norm. 
Their theoretical results in particular rely on functional limit theorems for the local time of diffusions on the circle.

\subsection{Statistical applications}
The concept of local time is deeply rooted in probability theory.
As indicated above, it however presents a very interesting object from a statistical point of view, too.
For another concrete motivation, let us specify again to the important class of ergodic diffusion process solutions of SDEs of the form \eqref{intro} with invariant density $\rho_b$.
Given a set of observations of the solution of \eqref{intro} with unknown drift $b\colon \R\to\R$, natural statistical questions concern the estimation of $b$ and of the invariant density $\rho_b$. 
In fact, in view of the basic relation $b=(\sigma^2\rho_b)'/(2\rho_b)$, both tasks are obviously related. 
Note that continuous observations can identify the diffusion coefficient $\sigma^2$. 
Therefore, it is considered to be known, and the focus is on estimation of the drift coefficient $b$.

\paragraph{Invariant density estimation via local time}
Alternatively to \eqref{eq:tanaka}, $(\L_t^a(\X))_{t\geq 0}$ may be introduced via the following approximation result, holding a.s.~for every $a\in\R$ and $t\geq 0$,
\[
L_t^a(\X)\ =\ \lim_{\ep\to0}\frac{1}{\ep}\int_0^t\mathds{1}\{a\leq X_s\leq a+\ep\}\d\langle X\rangle_s.
\]
This representation now already suggests the meaningful interpretation of the local time as the derivative of an empirical distribution function.
Assuming that a continuous record of observations $(X_t)_{0\leq t\leq T}$ of the solution of \eqref{intro} is available, it thus appears natural to use local time for constructing an estimator $\rho_t^\circ$ of $\rho_b$ by letting
\begin{equation}\label{eq:rhocir}
\rho_t^\circ(a)\ :=\ \frac{L_t^a(X)}{t\sigma^2(a)}, \quad a\in\R.
\end{equation}
One might tackle the question of quantifying the quality of the estimator $\rho_t^\circ$ wrt the $\sup$-norm risk, e.g., by deriving upper bounds on the $p$-th ($p\geq1$) moments
\[
\E\left[\left(\sup_{a\in\R}\left|\rho_t^\circ(a)-\rho_b(a)\right|\right)^p\right]\ =\ \E\left[\left\|\frac{L_t^\bullet(X)}{t\sigma^2}-\rho_b\right\|_\infty^p\right].
\]
Local time thus presents an object of its own statistical interest. 
The corresponding investigation is subject of Section \ref{sec:tp}.

\paragraph{Kernel invariant density estimation}
Apart from the treatment of the local time estimator in $\sup$-norm loss, the statistical relevance of Theorem \ref{theo:bernsup} -- which deals with general empirical processes of a diffusion -- is demonstrated by a detailed study of the question of invariant density estimation via the \emph{kernel density estimator} (again in $\sup$-norm loss) and its relation to the local time density estimator in Section \ref{sec:tp}.
One clear advantage of the local time estimator $\rho_t^\circ$ introduced in \eqref{eq:rhocir} is that it allows for direct application of deep probabilistic results on diffusion local time. 
For example, weak convergence properties can be deduced in this way. 
At the same time, $\rho_t^\circ$ is merely of theoretical interest since its implementation in practice requires another approximation procedure.
One first step towards finding practically more feasible estimators is to replace $\rho_t^\circ$ by the standard kernel estimator
\begin{equation}\label{est:dens}
\rho_{t,K}(h)(x)\ :=\ \frac{1}{th}\int_0^tK\left(\frac{x-X_u}{h}\right)\d u,\quad x\in\R,
\end{equation}
$K\colon\R\to\R$ some smooth kernel function with compact support and $h>0$ some bandwidth.
The kernel density estimator outperforms the local time density estimator in various important aspects. 
First of all, from an applied perspective, working with the kernel density estimator serves as a universal, familiar approach to density estimation in all common models. 
For our particular diffusion framework, it is straightforward to extend the procedure to the case of discrete or multivariate observations. 
From a more theoretical perspective, the additional smoothness of the kernel estimator is desirable for investigations. 
The kernel density estimator can be viewed as a convolution operator applied to the local time. 
Interestingly, this smoothing is exactly what is required for proving the assertion on $\|\rho^\circ_t-\rho_b\|_\infty$ in Corollary \ref{cor:centlt}. 
Thus, our proof -- which makes use of the kernel density estimator -- is more natural than it might look at first sight. 
In addition, we show that our results on the moments of the supremum of empirical processes imply precise upper bounds on 
$\E\left[\|\rho_{t,K}(h)-\rho_b\|_\infty^p\right]$, $p\geq 1$. 
These upper bounds in particular verify that, in terms of performance in $\sup$-norm risk, the kernel density estimator with the universal bandwidth choice $t^{-1/2}$ is as good as the local time density estimator $\rho_t^\circ$.
Furthermore, we provide an in-depth analysis of the stochastic behaviour of $\left\|\rho_{t,K}(h)-\rho_t^\circ\right\|_\infty$ which in particular allows to transfer results for the local time estimator to the class of kernel estimators. 

\paragraph{Outlook: Application to adaptive (drift) estimation}
Beyond the question of invariant density estimation, another important statistical motivation for deriving the concentration inequalities in this paper is their application to adaptive estimation of the unknown drift coefficient $b$ in \eqref{intro}. 
This research goes beyond the scope of the present work and is dealt with in the preprint \cite{cacs18}. 
Using the presented results and techniques, we suggest a fully data-driven procedure which allows for rate-optimal estimation of the unknown drift wrt $\sup$-norm risk and, at the same time, yields an asymptotically efficient estimator of the invariant density of the diffusion. 
The procedure is based on Lepski's method for adaptive estimation. 
In \cite{cacs18}, we also deepen the analysis of the kernel density estimator started here. 
We derive a Donsker-type convergence result as it is relevant for the construction of (adaptive) confidence bands. 
Furthermore, we deal with the question of semi-parametric efficiency of the local time and the kernel density estimator in $\ell^\infty(\R)$. 
These contributions heavily rely on the exponential inequality for the $\sup$-norm difference between the local time and the kernel density estimator provided in Theorem \ref{theo:cath}. 
This result allows to transfer probabilistic knowledge on the local time to the more accessible and smoother kernel density estimator.

Apart from the apparent extensions to discrete observations of diffusion processes and multivariate state variables, further applications of the concentration inequalities derived in the current work could be found in the field of Bayesian statistical approaches, e.g., concerning supremum norm contraction rates.
Another very interesting application of the proposed martingale approximation approach to concentration inequalities concerns bifurcating Markov chains.
\cite{bihool17} construct adaptive nonparametric estimators of various quantities associated to bifurcating Markov chains.
Crucial ingredient for their proofs are Bernstein-type deviation inequalities which in particular can be applied to well localised but unbounded functions.
The corresponding findings are proven under a quite strong ergodicity assumption, and the authors suggest to use transportation-information inequalities for Markov chains for deriving similar results under more general conditions.
Since the idea of martingale approximation is applicable in the Markov chain set-up, too, there is a natural starting point for the machinery developed in this paper, providing another alternative approach to (even uniform) deviation inequalities for bifurcating Markov chains. 

\bigskip

\section{Exponential tail inequality for the supremum of the local time of continuous semimartingales}\label{sec:2}
Throughout this section, we work on a complete filtered probability space\linebreak
$(\Omega,\mathscr{F},(\mathscr{F}_t)_{t\geq0},\P)$, 
and we consider a continuous semimartingale $X$ with canonical decomposition $X=X_0 + M+V$.
Here, $X_0$ is an $\mathscr F_0$-measurable random variable, $M=(M_t)_{t\geq 0}$ denotes a continuous martingale with $M_0=0$ and $V=(V_t)_{t\geq 0}$ is a finite variation process with $V_0=0$.
To shorten notation, we will often abbreviate
\[\|Y\|_p\ :=\ \left(\E\left[|Y|^p\right]\right)\p,\quad\text{ for }Y\in L^p(\P),\ p\geq 1.\]
For proving concentration inequalities for generalised additive functionals of the semimartingale $X$,
we impose very general assumptions on the behaviour of the moments of the total variation of $V$ and the quadratic covariation of $M$.

\medskip

\begin{annahme}\label{M+V_upper_bound}
There exist deterministic functions $\phi_1\colon \R_+ \rightarrow\R_+$, $\phi_2\colon \R_+ \rightarrow\R_+$ such that, for any $p\geq1$,
\begin{equation}\label{eq:p}
\|X_0\|_p +\|X_t\|_p + \Big\|\int_0^t|\d V_s|\Big\|_p\ \leq\ p\phi_1(t),\qquad \left(\E\left[\langle M\rangle_t^{p/2}\right]\right)\p\ \leq\ \phi_2(t), \quad t>0.
\end{equation}
Here, $(\int_0^t|\d V_s|)_{t\geq 0}$ denotes the total variation process of $V$, and we write $|\d V_s|$ for integration with respect to the total variation measure of $V$.
Furthermore, we assume that
\[\lim_{t\to\infty}\phi_1(t)\ =\ \infty\quad\text{ and }\quad \phi_2(t)\ \leq\ \sqrt {\phi_1(t)}.
\]
\end{annahme}

With regard to our goal of proving tail estimates of the \emph{supremum} of stochastic processes, we are interested in finding upper bounds for all $p$-th moments of 
\[\boldsymbol{\sup_{a\in\R}}|L_t^a(\X)|\ =\ \|L_t^\bullet(\X)\|_\infty.\]
The derivation of such uniform bounds is rather involved and comprises several steps. 
While the complete proof has been deferred to the Appendix, it is instructive to sketch the main ideas now.
A natural starting point is given by Tanaka's formula. 
Using \eqref{eq:tanaka} and then \eqref{eq:p}, one obtains a decomposition of the local time process which allows to derive the upper bound 
\begin{equation}\label{eq:bound}
\left(\E\left[\left\|\L_t^\bullet(\X)\right\|_\infty^p\right]\right)\p
\ \leq\ 2p\phi_1(t) + 2\bigg(\E\bigg[\bigg(\sup_{a\in\mathbb Q}\mathds{1}\left\{\max_{0\leq s\leq t}|X_s|\geq |a|\right\}|\M^a_t|\bigg)^p\bigg]\bigg)\p,
\end{equation}
where $\M^a_t:=\int_0^t\mathds{1}\{X_s\leq a\}\d M_s$, $a\in\R$.
Dealing with the $\sup$-norm, it is crucial for the analysis to take into account the random, compact support of the local time in inequality \eqref{eq:bound}. 
The size of the support depends on the extremal behaviour of the semimartingale, i.e., if $a\in\supp(\L^\bullet_t(X))$, then necessarily $\max_{0\leq s \leq t}|X_s|\geq |a|$. 
This will allow to extend local arguments to the whole real line.

Coming back to \eqref{eq:bound}, the main task now consists in controlling the martingale part appearing in the last summand, and it is classical to use the BDG inequality in this respect. The best constant in the BDG inequality is of order $O(\sqrt p)$, and this fact plays an important role in our subsequent developments.
More precisely, Proposition 4.2 in \cite{bayo82} states that there exists a constant $\overline\bdg\geq 1$ such that, 
for any $p\geq 2$ and any continuous martingale $(N_t)_{t\geq0}$ with $N_0=0$, one has
\begin{equation}\label{4.2}
	\left(\E\left[\left(\sup_{0\leq s\leq t}\left|N_s\right|\right)^p\right]\right)\p\ \leq\ \overline\bdg\sqrt p\left(\E\left[\langle N\rangle_t^{p/2}\right]\right)\p.
\end{equation}
Consequently, whenever Assumption \ref{M+V_upper_bound} holds true, one obtains for any $p\geq 1$ 
\begin{equation}\label{eq:baryor}
\left(\E\left[\left(\sup_{0\leq s\leq t}\left|\M_s\right|\right)^p\right]\right)\p\ \leq\ \bdg\sqrt p\phi_2(t),\quad\text{with }\bdg:=\max\left\{1,\sqrt 2 \overline\bdg\right\},
\end{equation}
due to H\"older's inequality and \eqref{eq:p}.
The upper bound \eqref{eq:baryor} in particular allows to explore the tail behaviour of $(\M^a_t)_{a\in\R}$.
A chaining procedure then yields an upper bound on the expectation on the rhs of \eqref{eq:bound} in terms of entropy integrals. 
This chaining procedure has to be done locally first since -- in terms of the finiteness of covering numbers -- the corresponding metric structure is not well behaved on the whole real line. Therefore, compact intervals of fixed length are considered, and it is taken into account that the probability of the support of the local time exceeding certain levels is vanishing (see Figure \ref{T1}).
The following maximal inequality for the process $(X_s)_{s\in[0,t]}$ allows to control this probability.
Its short proof nicely illustrates the basic idea of how to exploit the moment bounds given in \eqref{eq:p}.

\begin{figure}
\begin{center}
\includegraphics[trim = 0mm 20mm 0mm 20mm, clip, width=0.8\textwidth]{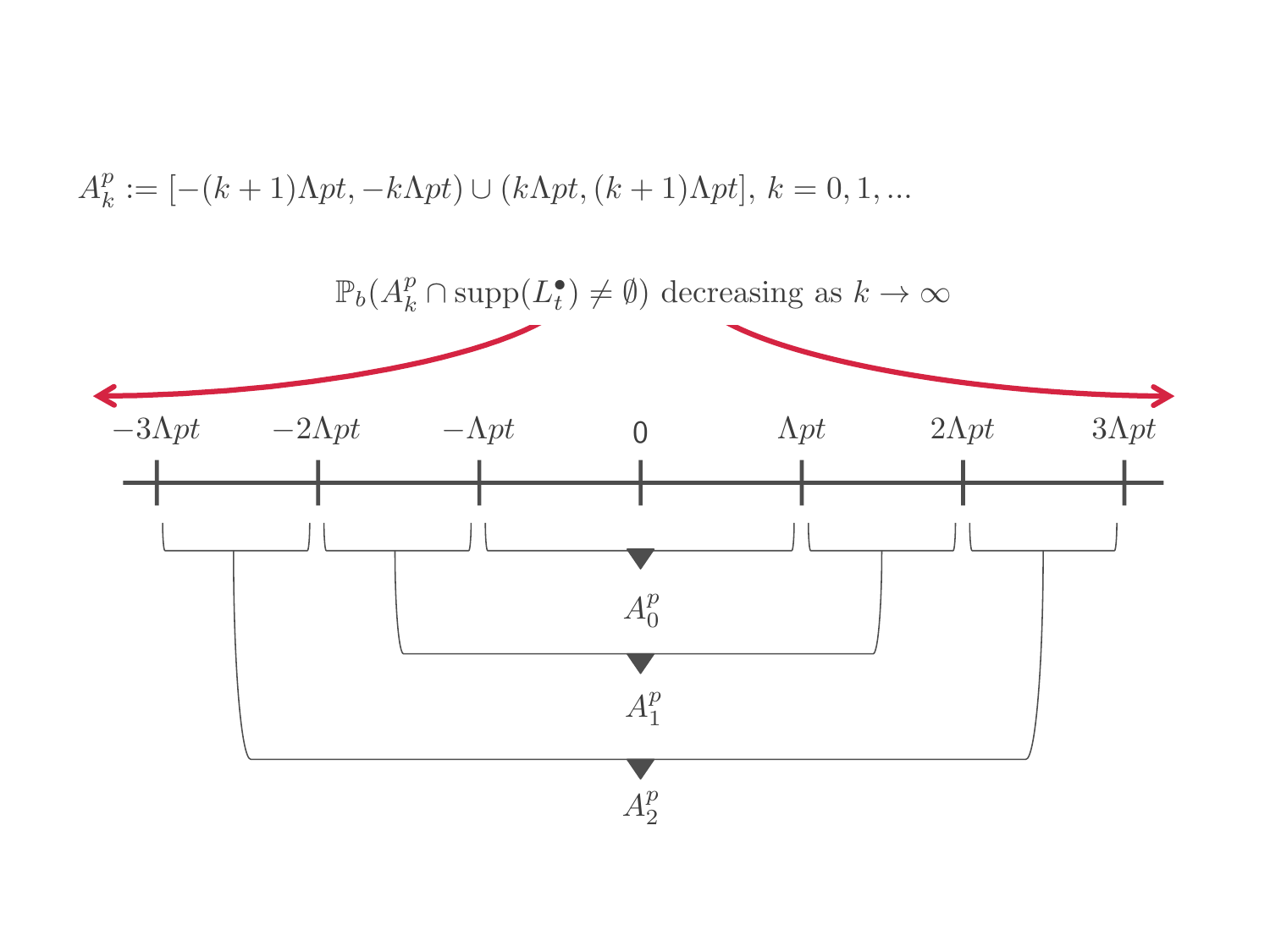}
\caption{Localisation procedure}\label{T1}
\end{center}
\end{figure}

\medskip

\begin{lemma}[Maximal inequality for $\X$]\label{maximal_inequality_X}
Under Assumption \ref{M+V_upper_bound}, it holds for any $u\geq 1$
\begin{equation}\label{maxineqX}
\P\left(\max_{0\leq s \leq t}|X_s|\geq \e\left(u\phi_1(t)+\bdg\sqrt u\phi_2(t)\right)\right)\ \leq\ \e^{-u}.
\end{equation}
\end{lemma}
\begin{proof}
Note that 
\[
	\max_{0\leq s \leq t}|X_s|\ \leq\ |X_0| + \int_0^t|\d V_s| + \max_{0\leq  s\leq t} |M_t|.
\]
Consequently, using \eqref{eq:p} and \eqref{eq:baryor}, for any $p\geq 1$,
\[\Big\|\max_{0\leq s \leq t}|X_s|\Big\|_p\ \leq\ \phi_1(t)p +  \bdg\sqrt p\phi_2(t).\]
Lemma \ref{moment_tail_lemma} from Appendix \ref{app:A} then gives \eqref{maxineqX}.
\end{proof}

In particular, the maximal inequality \eqref{maxineqX} provides the final ingredient for verifying the main result of this section. 
Its complete proof is given in Appendix \ref{App:B}.

\medskip

\begin{theorem}\label{theo:lt}
Consider a continuous semimartingale $X$ with canonical decomposition $X=X_0+M+V$, and grant Assumption \ref{M+V_upper_bound}.
Then, there exists a positive constant $\kappa$ (not depending on $p$) such that, for any $p\geq 1$,
\[
\left(\E\left[\left\|\L_t^{\bullet}(X)\right\|_\infty^p\right]\right)^{\frac{1}{p}}
\ \leq\ \kappa\left(p\phi_1(t)+\sqrt{p}\phi_2(t)+\left(\sqrt{\phi_1(t)}+ \sqrt{\phi_2(t)}\right)\log(2p\Lambda(t))\right),
\]
where $\Lambda(t):=\e\left(\phi_1(t)+\bdg\phi_2(t)\right)$.
Consequently, for any $u\geq 1$,
\[
\P\left(\|\L_t^\bullet(X)\|_\infty\geq \e \kappa\left(u\phi_1(t) + \sqrt{u}\phi_2(t) + \left(\sqrt{\phi_1(t)} + \sqrt{\phi_2(t)}\right)\log(2u\Lambda(t))\right)\right)
\ \leq\ \e^{-u}.
\]
\end{theorem}

\bigskip

\section{Uniform concentration of empirical processes of continuous semimartingales}\label{sec:3}
In Section \ref{sec:2}, we focused on analysing the $\sup$-norm of the local time. 
Rephrasing the problem, we realise why the proof of Theorem \ref{theo:lt} is a blueprint for investigating a much more general setting. 
Letting $\mathcal F:=\left\{\mathds{1}_{(-\infty,a]}(\cdot):\, a\in\R\right\}$, Tanaka's formula and equation \eqref{eq:bound} reveal the core of the investigation: 
It consists in controlling 
\[\sup_{a\in\R}\M_t^a\ =\ \sup_{a\in\R}\int_0^t\mathds{1}\left\{X_s\leq a\right\}\d M_s\ =\ \sup_{f\in\mathcal F}\int_0^t f(X_s)\d M_s.\]
Thus, the supremum of the process can be analysed within the framework of empirical processes and related concepts.
The purpose of this section is to extend the study from the specific case of local time to additive functionals of the form $\sup_{f\in\mathcal F}\int_0^tf(X_s)\d s$ and further to stochastic integrals $\sup_{f\in\mathcal F}\int_0^tf(X_s)\d X_s$. 

We start by investigating empirical processes of some continuous semimartingale $X$ of the form
\begin{equation}\label{eq:g}
(\G_t^{b_0}(f))_{f\in\FF}\ :=\ \left(\frac{1}{\sqrt t}\int_0^t \left(f(X_u)b_0(X_u)-\E[f(X_0)b_0(X_0)]\right)\d u\right)_{f\in\FF},\quad t>0,
\end{equation}
indexed by a countable family $\FF\subset L^2(\lebesgue)$, $\lebesgue$ denoting the Lebesgue measure, and for a function $b_0\colon \R\to\R$ satisfying 
\begin{equation}\label{eq:b0}
|b_0(x)|\ \leq\ C(1+|x|^\eta),
\end{equation} 
$\eta\geq 0$, $C\geq 1$ some fixed constants.
The main idea for deriving concentration inequalities is to use the technique of martingale approximation which was already introduced in Section \ref{sec:intro} (cf.~\eqref{eq:martapp}) in a more systematic manner.
While Theorem \ref{theo:lt} for the local time concerns the supremum taken over the whole real line, we now turn to investigating suprema over general (possibly infinite-dimensional) function classes. 
For any semi-metric space $(\mathcal F, d)$, denote by $N(u,\FF,d)$, $u>0$, the covering number of $\FF$ wrt $d$, i.e., the smallest number of balls of radius $u$ in $(\FF,d)$ needed to cover $\FF$. Furthermore, we introduce  
\[E(\FF,d,\alpha)\ :=\ \int_0^\infty \left(\log N(u,\FF,d)\right)\a\d u,\quad \alpha>0.\]
With regard to the indexing classes of functions $\FF$ in \eqref{eq:g}, we impose the following basic conditions.

\medskip

\begin{annahme}\label{ann:1}
$\FF$ is a countable class of real-valued functions satisfying, for some fixed constants $\U,\V>0$,
\[\sup_{x\in\R}|f(x)|\ \leq\ \U,\quad \sup_{f\in\FF}\|f\|_{L^2(\lebesgue)}\ \leq \V.\]
In addition, all $f\in\FF$ have compact support with
\[\supp(f)\subset [x_f,x^f], \text{ where }|x^f-x_f|\leq \mathcal S \text{ and }\V\leq\sqrt\SS,\text{ for some }x_f< x^f,\,\mathcal S>0.\]
\end{annahme}

\medskip

\begin{annahme}\label{ann:3}
$\FF$ is a countable class of real-valued functions such that there exist constants $\e^2<\AA<\infty$ and $v\geq 2$ such that, for any probability measure $\mathbb{Q}$,
\begin{equation}\label{eq:ent}
\forall\ep\in(0,1),\quad N\left(\ep,\FF,\|\cdot\|_{L^2(\mathbb{Q})}\right)\ \leq \ (\AA/\ep)^v.
\end{equation}
\end{annahme}

\medskip

Throughout the sequel, $\mom>0$ denotes a constant satisfying
$\|X_0\|_p=\left(\E\left[|X_0|^p\right]\right)\p\leq p\mom$, $p\geq 1$. The existence of such a constant follows from Assumption \ref{M+V_upper_bound}.
Furthermore, we use the notation $\sup_{f\in\FF}|\G_t(f)|=:\|\G_t\|_\FF$.
\medskip

\begin{theorem}\label{theo:1}
Let $\X$ be a continuous semimartingale as in Assumption \ref{M+V_upper_bound}, and let $b_0\colon\R\to\R$ be a function satisfying \eqref{eq:b0} for some constants $\eta\geq 0, C\geq 1$.
Suppose that the function class $\FF$ satisfies Assumption \ref{ann:1}, and define $\G_t^{b_0}(\cdot)$ according to \eqref{eq:g}.
Assume further that any $f\in\FF$ admits a martingale approximation 
\begin{equation}\label{martrep}
\G^{b_0}_t(f)\ =\ t^{-1/2}\Ma_t^{f}\ +\ t^{-1/2}\rd_t^{f},\quad t>0,
\end{equation}
for which there exist constants $\Psi_1,\Psi_2$ and some $\alpha>0$ such that, for any $f,g\in\FF$,
\begin{align}\label{martrep2}\begin{split}
\left(\E\left[|\Ma_t^f|^p\right]\right)\p&\leq\ \Psi_1\sqrt t p\a\|f\|_{L^2(\lebesgue)}, \hspace*{3em}\left(\E\left[\sup_{f\in\FF}|\rd^f_t|^p\right]\right)\p\ \leq\ \Psi_2p,\\
\left(\E\left[|\Ma_t^f-\Ma_t^g|^p\right]\right)\p&\leq\ \Psi_1\sqrt t p\a\|f-g\|_{L^2(\lebesgue)}.
\end{split}
\end{align}
For $k\in\N_0$ and fixed $p\geq 1$, define 
\begin{equation}\label{ikfk}
\begin{array}{r@{}l}
I_k\ &:=\ \big(-2(k+1)p\Lambda(t),\ -2kp\Lambda(t)\big]\cup \big(2kp\Lambda(t),\ 2(k+1)p\Lambda(t)\big] \oplus [-\mathcal S, \mathcal S],\\ [4pt]
\mathcal{F}_k\ &:=\ \left\{f\in\mathcal F\colon \supp(f)\subset I_k\right\},
\end{array}
\end{equation} with
\begin{equation}\label{def:lamb}
\Lambda(t)\ :=\ \max\left\{\lambda\e\left(\phi_1(t)+\bdg\phi_2(t)\right),1\right\}
\end{equation} 
and $\lambda>1$ such that $\max\{\mathcal{S},\e\mom\}< p\Lambda(t)$, for any $p,t\geq 1$. 
Then, for any $t,p\geq 1$, whenever
\begin{equation}\label{sumentropyintegral}
\sum_{k=0}^{\infty}E(\FF_k,\e\Psi_1\|\cdot\|_{L^2(\lebesgue)},\alpha) \exp\left(-\frac{k}{2}\right)\ <\ \infty,
\end{equation}
it holds
\begin{align*}
\left(\E\left[\|\G^{b_0}_t\|_{\mathcal{F}}^p\right]\right)^{\p}
&\leq\ C_\alpha\sum_{k=0}^{\infty}  E(F_k,\e\Psi_1\|\cdot\|_{L^2(\lebesgue)},\alpha) \exp\left(-\frac{k}{2}\right) + 6\Psi_1 (2p)\a\V+2\frac{\Psi_2p}{\sqrt t}\\
&\hspace*{12em}+ \sqrt t C\U(1+2\eta\mom)^{\eta}\exp\left(-\frac{\Lambda(t)}{2\e\mom}\right).
\end{align*}
\end{theorem}

A few comments on the above result are in order.

\medskip

\begin{remark}
\begin{itemize}
\item[(a)]
It will be shown that there exists a broad class of ergodic diffusion processes admitting a decomposition of the form \eqref{martrep}, with moments satisfying \eqref{martrep2}.
In most cases, it is not that difficult to bound the moments of the remainder term $\rd_t^f$, and usually the corresponding arguments already imply the \emph{uniform} moment bounds required in \eqref{martrep2}.
The analysis of the martingale part $\Ma_t^f$ is more challenging. 
Under the given assumptions, it suffices however to derive \emph{non-uniform} upper bounds on $\|\Ma_t^f\|_p$.
Theorem \ref{theo:1} then allows to translate these bounds into bounds on $\|\boldsymbol{\sup_{f\in\FF}}|\G_t^{b_0}(f)|\|_p$.
\item[(b)]
Assumption \eqref{sumentropyintegral} is a very weak one. 
In fact, we will show that the conditions of Theorem \ref{theo:1} and Assumption $\ref{ann:3}$ on the function class $\FF$ imply that \eqref{sumentropyintegral} holds true for $\alpha \in\{2/3,1,2\}$ (cf.~Lemma \ref{EntropyIntegralLebesgueMetric} in Appendix \ref{app:A}).
Whenever $E(\FF_k,\e\Psi_1\|\cdot\|_{L^2(\lebesgue)},\alpha)$ can be upper bounded independently of $k$, say $E(\FF_k,\e\Psi_1\|\cdot\|_{L^2(\lebesgue)},\alpha)\leq\EE(p,\alpha)$ for all $k\in\N_0$ and some finite constant $\EE(p,\alpha)>0$, Theorem \ref{theo:1} yields 
\[
\left(\E\left[\|\G^{b_0}_t\|_{\mathcal{F}}^p\right]\right)^{\p}
\ \leq\ 3C_\alpha\EE(p,\alpha)+ 6\Psi_1(2p)\a\V+2\frac{\Psi_2p}{\sqrt t}+\sqrt t C\U(1+2\eta\mom)^{\eta}\e^{-\frac{\Lambda(t)}{2\e\mom}}.
\]
Lemma \ref{EntropyIntegralLebesgueMetric} provides such an upper bound $\EE(p,\alpha)$ for $\alpha \in\left\{2/3,1,2\right\}$. 
Furthermore, in a lot of interesting instances (e.g., local time or the statistical application in Section \ref{sec:tp}), the function class $\mathcal F$ is translation invariant, i.e., for any constant $c\in\R$, $f\in\mathcal F$ implies that $f(\cdot + c)\in\mathcal F$. 
In that case, $E(\mathcal F_k,\e\Psi_1\|\cdot\|_{L^2(\lebesgue)},\alpha)$ does \emph{not} depend on $k$, and the finiteness of this quantity entails \eqref{sumentropyintegral}.
\item[(c)] 
Instead of assuming $X$ to be a continuous semimartingale fulfilling the moment bounds \eqref{eq:p} in Assumption \ref{M+V_upper_bound}, one could also work with other classes of processes satisfying a maximal inequality as in Lemma \ref{maximal_inequality_X} and allowing for a martingale approximation with moment bounds as in \eqref{martrep2}.
\end{itemize}
\end{remark}

\medskip

\begin{proof}[Proof of Theorem \ref{theo:1}]
Fix $p\geq 1$. 
The definition of $\Lambda(t)$ (cf.~\eqref{def:lamb}) implies for any $k\in\N$, setting $u=kp$,
\[\e\left(u\phi_1(t)+ \bdg\sqrt u\phi_2(t)\right)\ 
\leq\ kp\e\left(\phi_1(t)+\bdg\phi_2(t)\right)\ \leq\ kp\Lambda(t),\]
and consequently, according to Lemma \ref{maximal_inequality_X},
\begin{equation*}
\P\left(\max_{0\leq s\leq t}|X_s|>kp\Lambda(t)\right)\ \leq\ \exp\left(-kp\right).
\end{equation*}
Furthermore, since $\|X_0\|_p\ \leq\ p\mom$, Lemma \ref{moment_tail_lemma} yields
\begin{equation*}
\P\left(|X_0|\ \geq\ \e\mom\ u\right)\ \leq\ \exp(-u),\quad u\geq 1.
\end{equation*}
Set $A_f\ :=\ \left\{\exists\ s\in[0,t] \,\text{ such that } X_s\in \supp(f)\right\}$, and note that, for $f\in\mathcal{F}_k$, $k\in\N$, 
\[A_f\ \subset\ \left\{\max_{0\leq s \leq t}|X_s|\geq kp\Lambda(t)\right\}\ =:\ A_k,\] 
since, for any $x\in\supp(f)$,
$|x|\ \geq\ 2kp\Lambda(t) - \mathcal S\ \geq\ 2kp\Lambda(t) - kp\Lambda(t)\ =\ kp\Lambda(t)$.
Let $\mathcal{F}_0\c:=\cup_{k=1}^\infty \FF_k$. 
Note that, for $f\in\mathcal{F}_0\c$, 
\begin{align*}
\left|\E\left[f(X_0)b_0(X_0)\right]\right|&\leq\ C\|f\|_\infty \E\left[(1+|X_0|)^\eta\mathds{1}\left\{|X_0|\geq k p\Lambda(t)\right\}\right]\\
&\leq\ C\|f\|_\infty \left(\E\left[(1+|X_0|)^{2\eta}\right]\right)^{1/2}\ \left(\P\left(|X_0|\geq k p\Lambda(t)\right)\right)^{1/2}\\
&\leq\ C\|f\|_\infty(1+2\eta\mom)^{\eta}\exp\left(-\frac{\Lambda(t)}{2\e\mom}\right).
\end{align*}
Consequently, it holds 
$\sqrt t\left|\E\left[f(X_0)b_0(X_0)\right]\right|\leq \sqrt t C\U(1+2\eta\mom)^{\eta}\exp\left(-\frac{\Lambda(t)}{2\e\mom}\right)$.
We thus obtain the following decomposition:
\begin{align*}
\left(\E\left[\|\G_t^{b_0}\|_{\mathcal{F}}^p\right]\right)^{\p}
&\leq\ \left(\E\left[\|\G_t^{b_0}\|_{\mathcal{F}_0}^p\right]\right)^{\p}\ + \ \left(\E\left[\|\G_t^{b_0}\ \mathds{1}(A_f)\|_{\mathcal{F}_0\c}^p\right]\right)^{\p}\ + 
\ \left(\E\left[\|\G_t^{b_0}\ \mathds{1}(A_f\c)\|_{\mathcal{F}_0\c}^p\right]\right)^{\p}\\
&=\ \left(\E\left[\|\G_t^{b_0}\|_{\mathcal{F}_0}^p\right]\right)^{\p}\ +\ \left(\E\left[\|\G_t^{b_0}\ \mathds{1}(A_f)\|_{\mathcal{F}_0\c}^p\right]\right)^{\p}+\sqrt t\|\E\left[f(X_0)b_0(X_0)\right]\|_{\mathcal{F}_0\c}\\
&\leq\ \left(\E\left[\|\G_t^{b_0}\|_{\mathcal{F}_0}^p\right]\right)^{\p}\ + \
\left(\E\left[\|\G_t^{b_0}\ \mathds{1}(A_f)\|_{\mathcal{F}_0^c}^p\right]\right)^{\p}\\
\\
&\hspace*{1.5em} +\sqrt t C\U(1+2\eta\mom)^{\eta}\e^{-\frac{\Lambda(t)}{2\e\mom}}.
\end{align*}
Regarding the first two terms in the last display, note that
\begin{align*}
\left(\E\left[\|\G_t^{b_0}\|_{\mathcal{F}_0}^p\right]\right)^{\p}
&\leq\ \frac{1}{\sqrt{t}}\left\{\left(\E_b\left[\|\Ma_t^f\|_{\mathcal{F}_0}^p\right]\right)^{\p}+\left(\E\left[\|\rd_t^{f}\|_{\mathcal{F}_0}^p\right]\right)\p\right\},\\
\left(\E\left[\|\G_t^{b_0}\ \mathds{1}(A_f)\|_{\mathcal{F}_0\c}^p\right]\right)^{\p}
&\leq\ \frac{1}{\sqrt{t}} \left\{\left(\E\left[\|\Ma_t^f\ \mathds{1}(A_f)\|_{\mathcal{F}_0\c}^p\right]\right)\p+\left(\E\left[\|\rd_t^{f}\|_{\mathcal{F}_0\c}^p\right]\right)\p\right\}.
\end{align*}
Thus, 
\begin{equation}\label{eq:a1}
\left(\E\left[\|\G^{b_0}_t\|_{\mathcal{F}}^p\right]\right)^{\p}\ \leq\ A\ +\ B\ +\ \sqrt t C\U(1+2\eta\mom)^{\eta}\exp\left(-\frac{\Lambda(t)}{2\e\mom}\right),
\end{equation} 
where
\begin{align*}
A\ :=\ \frac{1}{\sqrt{t}}\left\{\left(\E\left[\|\Ma_t^f\|_{\mathcal{F}_0}^p\right]\right)^{\p} + \left(\E\left[\|\Ma_t^f\mathds{1}(A_f)\|_{\mathcal{F}_0\c}^p\right]\right)^{\p}\right\},\quad
B\ :=\ \frac{2}{\sqrt{t}}\left(\E\left[\|\rd_t^{f}\|_{\mathcal{F}}^p\right]\right)^{\p}.
\end{align*}
Assumption \eqref{martrep2} implies that, for any $f, g\in\FF_k$, $\|\Ma_t^f\|_p\leq \Psi_1\sqrt t p^{1/\alpha}\V$, and the following tail estimate,
\[\P\left(|t^{-1/2}(\Ma_t^f-\Ma_t^g)| \geq  d_2(f,g)u\right)\ \leq\ \exp\left(-u^\alpha\right),\quad u\geq 1,\]
where $d_2(f,g):=\e\Psi_1\|f-g\|_{L^2(\lebesgue)}$.
Proposition \ref{thm:dirk} then yields, for any $k\in\N_0$, $q\geq 1$,
\begin{align}\nonumber
\frac{1}{\sqrt{t}}\left(\E\left[\|\Ma_t^f\|_{\mathcal{F}_k}^q\right]\right)^{\frac{1}{q}}
&\leq\ C_\alpha \int_0^\infty \left(\log N(u,\mathcal F_k, d_2)\right)\a\d u + \frac{2}{\sqrt t}\sup_{f\in\mathcal F_k}\|\Ma_t^f\|_q\\\label{eq:a2}
&\leq\ C_\alpha \int_0^\infty \left(\log N(u,\mathcal F_k, d_2)\right)\a\d u + 2\Psi_1 q\a\V,
\end{align}
and, for all $k\in\N$,
\begin{align}\nonumber
\frac{1}{\sqrt t}\left(\E\left[\|\Ma_t^f\mathds{1}(A_f)\|^p_{\mathcal{F}_0\c}\right]\right)^{\p}
&\leq\ \sum_{k=1}^\infty \frac{1}{\sqrt t}\left(\E\left[\|\Ma_t^f\mathds{1}(A_k)\|^p_{\mathcal{F}_k}\right]\right)^{\p}\\\nonumber
&\leq\ \sum_{k=1}^\infty \frac{1}{\sqrt t}\left(\E\left[\|\Ma_t^f\|^{2p}_{\mathcal{F}_k}\right]\right)^{\frac{1}{2p}} \P(A_k)^{\frac{1}{2p}}\\\nonumber
&\leq\ \sum_{k=1}^\infty \frac{1}{\sqrt t}\left(\E\left[\|\Ma_t^f\|^{2p}_{\mathcal{F}_k}\right]\right)^{\frac{1}{2p}} \exp\left(-\frac{k}{2}\right)\\\label{eq:a3}
&\leq\ \sum_{k=1}^{\infty }\left[C_\alpha \int_0^\infty \left(\log N(u,\mathcal F_k, d_2)\right)\a\d u\e^{-\frac{k}{2}}\right] + 4\Psi_1 (2p)\a\V.
\end{align}
Finally, the announced moment bound follows from \eqref{eq:a1}, \eqref{eq:a2}, \eqref{eq:a3} and \eqref{martrep2}. 
\end{proof}

\bigskip

\section{Concentration of measure and exponential inequalities for scalar ergodic diffusions}\label{sec:cmei}
The original motivation for the present study was the question of deriving exponential inequalities for diffusion processes and associated additive functionals as they are constantly used for investigating (adaptive) statistical procedures. 
The current analysis has a much wider scope, and the results and methods of proof actually apply in a much more general framework. 
However, for clarity of presentation and in order not to lose the main ideas, we focus in the sequel on a specific class of diffusion processes. 
The results of this section take up those established in Section \ref{sec:2} (for local times) and Section \ref{sec:3} (for empirical processes) for the specific diffusion setting.
In Section \ref{subsec:4.3}, we even go one step further and establish a concentration result for generalised empirical processes that involve stochastic integrals.
We start with introducing our basic class of diffusion processes.

\medskip

\begin{definition}\label{def:B}
Let $\sigma\in \operatorname{Lip}_{\operatorname{loc}}(\R)$ and assume that, for some constants $\overline{\nu},\underline\nu\in(0,\infty),$  $\sigma$ satisfies $\underline\nu\leq \left|\sigma(x)\right|\leq \overline\nu$ for all $x\in\R$ . For fixed constants $A,\gamma>0$ and $\C\geq 1$, define the set $\Sigma=\Sigma(\C,A,\gamma,\sigma)$ as
\begin{equation}\label{eq:sigma}
\Sigma\ :=\ \Big\{b \in \operatorname{Lip}_{\operatorname{loc}}(\R)\colon|b(x)| \leq\C(1+|x|),\ \forall|x|>A\colon \frac{b(x)}{\sigma^2(x)}\operatorname{sgn}(x)\leq -\gamma\Big\}.
\end{equation}
\end{definition}

Given $\sigma$ as above and any $b\in\Sigma$, there exists a unique strong solution of the SDE \eqref{intro} with ergodic properties and invariant density
\begin{equation}\label{eq:invdens}
\rho(x)=\rho_b(x)\ :=\ \frac{1}{C_{b,\sigma}\sigma^2(x)}\ \exp\left(\int_0^x\frac{2b(y)}{\sigma^2(y)}\d y\right),\quad x\in\R,
\end{equation}
with $C_{b,\sigma}:=\int_\R\frac{1}{\sigma^2(u)}\ \exp\left(\int_0^u\frac{2b(y)}{\sigma^2(y)}\d y\right)\d u$ denoting the normalising constant.  The invariant measure of the corresponding distribution and its distribution function will be denoted by $\mu=\mu_b$ and $F=F_b$, respectively, and we assume that the process is started in the equilibrium, i.e., $\xi\sim\mu_b$.

\smallskip

Our assumptions on the diffusion characteristics already impose some regularity on the invariant density $\rho_b$. 
More precisely, for any $b\in\Sigma(\C,A,\gamma,\sigma)$, $\sigma^2\rho_b$ is continuously differentiable and there exists a constant $\ML >0$ (depending only on $\C,A,\gamma,\underline\nu, \overline\nu$) such that 
\begin{equation}\label{eq:reg_rho_b}
\sup_{b\in\Sigma(\C,A,\gamma,\sigma)}\max\left\{\|\rho_b\|_{\infty},\|(\sigma^2\rho_b)'\|_\infty\right\}\ <\ \ML 
\end{equation}
and, for any $\theta>0$, we have 
$\sup_{b\in\Sigma(\C,A,\gamma,\sigma)}\sup_{x\in\R}\left\{|x|^{\theta}\rho_b(x)\right\} <\infty$.
The analysis of the moments of functionals of the process $X$ relies on upper bounds for the moments of the invariant measure. 
For any diffusion process $X$ as in Definition \ref{def:B}, it holds 
\begin{equation}\label{moment_bounds_invariant_measure}
\sup_{b\in\Sigma(\C,A,\gamma,\sigma)}\|X_0\|_p\ = \ \sup_{b\in\Sigma(\C,A,\gamma,\sigma)} \left(\E_b\left[|X_0|^p\right]\right)\p\ \leq\ \mom p,\quad p\geq 1,
\end{equation}
(cf.~Lemma \ref{X_0_Mom} in Appendix \ref{app:A}) for some positive constant $\mom$.
The above estimates will be used in the sequel without further notice.

\medskip

\begin{remark}\label{rem:bern}
A natural approach for analysing the supremum of processes of the form
\[\frac{1}{t}\int_0^tg(X_s)\d s\quad\text{ or } \quad\frac{1}{t}\int_0^t g(X_s)\d X_s\]
over entire function classes consists in making use of well-known concentration results for additive diffusion functionals.
For any nice diffusion $\X$ fulfilling Poincar\'e's inequality, it is actually known that, for any \emph{bounded} function $g\colon\R\to\R$, one has a Bernstein-type tail estimate of the form
\begin{equation}\label{ineq:bern}
\P\left(\frac{1}{t}\int_0^t\left(g(X_s)-\E[g(X_0)]\right)\d s>r\right)\ \leq\ \exp\left(-\frac{tr^2}{2(\Var(g)+c_P\|g\|_\infty r)}\right),
\end{equation}
for $t,r>0$ and $c_P$ denoting the Poincar\'e constant.
Given any class $\GG$ of bounded functions $g\colon\R\to\R$ fulfilling \eqref{ineq:bern}, the above inequality implies that the process $(\G_t(g))_{g\in\GG}$ exhibits a mixed tail behaviour wrt the metrics $d_1(g,g'):=\|g-g'\|_\infty$ and $d_2(g,g'):=\Var(g-g')$.
Chaining procedures as they are used, e.g., for proving Theorem \ref{theo:lt} then can be applied to obtain upper bounds of the form
\begin{equation}\label{eq:entbern}
\left(\E\left[\left(\sup_{g\in\GG}|\G_t(g)|\right)^p\right]\right)\p\ \lesssim\ \frac{1}{\sqrt t}\int_0^\infty\log N(\ep,\GG,d_1)\d\ep
+\int_0^\infty\sqrt{\log N(\ep,\GG,d_2)}\d\ep+\sqrt p+\frac{p}{\sqrt t}.
\end{equation}
However, for any bounded $g\in\GG$, one can also derive a decomposition of the form \eqref{eq:martapp} where both the martingale part $M_t(g)$ and the remainder term $R_t(g)$ can be controlled similarly to the local time case. 

We do not want to restrict to bounded drift terms $b\colon\R\to\R$. 
For analysing term $(\mathbf{I})$ in \eqref{def:A}, one thus actually requires results for \emph{unbounded} functions $g=fb$.
Using the method of transportation-information inequalities, \cite{ggw10} establish Bernstein-type concentration inequalities in the spirit of \eqref{ineq:bern} for unbounded functions $g\colon\R\to\R$. 
In principle, one might then deduce upper bounds similarly to \eqref{eq:entbern}.
Note however that the results of \cite{ggw10} apply only to a restricted class of diffusion processes.
Furthermore, it is far from clear how the corresponding entropy integrals can be controlled, not to say the finiteness of the rhs of \eqref{eq:entbern} is not at all clear.
\end{remark}

\smallskip

In view of the aforementioned obstacles, we return to the alternative approach of proving concentration results via martingale approximation.
In the sequel, we will specify the components of the decomposition \eqref{eq:martapp} and derive upper bounds on the moments of the martingale and the remainder term for a broad class of ergodic diffusion processes.

\subsection{Moment bounds and tail estimates for diffusion local time}\label{mom:difflt}
We start with revisiting our result on local time and specifying it for the case of diffusion processes as introduced in Definition \ref{def:B}.
Thus, we consider the diffusion local time process $(L_t^a(X))_{a\in\R}$, $t\geq 0$, which is continuous in $a$ and $t$. 

\paragraph{Bounding the moments of $\|L_t^\bullet\|_\infty$ by means of Theorem \ref{theo:lt}}
In order to deduce a result by means of Theorem \ref{theo:lt}, we first argue that Assumption \ref{M+V_upper_bound} is satisfied for any process $X$ as in Definition \ref{def:B}.
Indeed, the finite variation part in this set-up is given by the integrated drift term, i.e., $V_t=\int_0^t b(X_s)\d s$.
We thus obtain for the total variation process $\int_0^t|\d V_s|\leq \int_0^t|b(X_s)|\d s$  $\forall \, t\geq 0$. 
From the moment bounds of the invariant measure \eqref{moment_bounds_invariant_measure} and the at-most-linear-growth condition on $b\in\Sigma(\C,A,\gamma,\sigma)$, one might deduce that
\begin{eqnarray*}
\|X_0\|_p +\|X_t\|_p + \Big\|\int_0^t|\d V_s|\Big\|_p\ \leq\ 2\mom p+t\mathbb C(1+\mom p)\ \leq\ 4 pt\mathbb C(1+\mom).
\end{eqnarray*}
Furthermore, $\left(\E_b[\langle\int_0^\bullet \sigma(X_s)\d W_s \rangle_t^{p/2}]\right)^{1/p}\leq\overline\nu\sqrt t$. 
Thus, setting $\phi_1(t):=\max(4\mathbb C(1+\mom),\overline\nu^2)t$ and $\phi_2(t):=\overline\nu\sqrt t$, Assumption \ref{M+V_upper_bound} is fulfilled.
The function $t\mapsto\Lambda(t)$ from Theorem \ref{theo:lt} and Theorem \ref{theo:1} takes the form
\[
\Lambda(t)\ :=\ \lambda\e\left(\max(4\mathbb C(1+\mom),\overline\nu^2)t+\bdg\overline \nu\sqrt t\right),
\]
with $\lambda>1$ such that $\max\{\mathcal S,\e \mom\}<\lambda \e (\max(4\mathbb C(1+\mom),\overline\nu^2)+\bdg\overline\nu).$ 
Letting \[\Lambda\ :=\ \lambda \e (\max(4\mathbb C(1+\mom),\overline\nu^2)+\bdg\overline\nu),\] it holds $\Lambda(t)\leq \Lambda t$, $t\geq 1$, and all the previous proofs also work for $\Lambda t$ instead of $\Lambda(t)$ which we use in the following without further notice. 
Given these estimates, Corollary 9.10 in \cite{legall16} now gives, for any $a\in\R$, $p\geq 1$ and $t>0$,
\[
\left(\E_b\left[\left(L_t^a(X)\right)^p\right]\right)\p \ \leq\ \tilde c_p\left(pt+\sqrt t\right),
\]
$\tilde c_p$ some (unspecified) positive constant depending on $p$.
Application of Theorem \ref{theo:lt} yields the $\sup$-norm counterpart, namely, the following result for the supremum of diffusion local time. 

\medskip

\begin{corollary}\label{cor:local_time} Let $X$ be a diffusion process as in Definition \ref{def:B}. 
Then, there is a positive constant $\kappa $ (not depending on $p$) such that, for any $p,u,t\geq 1$,
\begin{align*}
\sup_{b\in\Sigma(\C,A,\gamma,\sigma)}\left(\E_b\left[\left\|L_t^\bullet(X)\right\|_\infty^p\right]\right)^{\frac{1}{p}}&\leq \kappa\left(pt + \sqrt{pt} + \sqrt t
\log t\right),\\
\P_b\left(\left\|L_t^\bullet(X)\right\|_\infty\geq \e\kappa\left(u t + \sqrt{ut}+ \sqrt t\log t\right)\right)&\leq\exp(-u).
\end{align*}
\end{corollary}

\medskip

\subsection{Martingale approximation for additive functionals of diffusion processes}\label{sec:4.2}
We now specify our analysis of empirical processes $(\G_t^{b_0}(f))_{f\in\FF}$ as introduced in \eqref{eq:g} to the ergodic diffusion case.
Given some function class $\FF$, denote $\overline \FF:=\left\{g-h\colon g,h\in\FF\right\}$.

\medskip

\begin{proposition}\label{prop:mom}
Let $\X$ be a diffusion as in Definition \ref{def:B}.
Then, for any continuous function $b_0$ fulfilling $|b_0(x)|\leq C(1+|x|^\eta)$, $C,\eta\geq 0$ some fixed constants, and any class $\mathcal F$ of continuous functions $f\colon\R\to\R$ fulfilling Assumption \ref{ann:1}, there exists a representation
\begin{equation}\label{eq:martre}
\G^{b_0}_t(f)\ =\ t^{-1/2}\Ma^{f}_t\ + \ t^{-1/2}\rd_t^{f},\quad t> 0,
\end{equation}
satisfying, for any $f,g\in\mathcal F$, $\Ma_t^{f-g}=\Ma_t^f - \Ma_t^g$. In addition, for any $p\geq1$ and any $f\in\FF\cup\overline{\FF}$,
\begin{align}\tag{$\mathbf{I}$}
\begin{split}
\left(\E_b\left[|\Ma_t^{f}|^p\right]\right)\p&\leq\ (2p)^{\eta +1/2}\sqrt {t\mathcal S}\|f\|_{L^2(\lebesgue)}\ \overline\nu\bdg \left(1+(\mom\eta)^\eta\right)\overline{\Lambda}_{\operatorname{\prox}},\\
\left(\E_b\left[\left(\sup_{f\in\FF}|\rd^{f}_t|\right)^p\right]\right)\p &\leq\ p^{\eta+1}\mathcal S\ 4\max\left\{\mom^{\eta+1}, 1\right\}(\eta+1)^{\eta}\overline\Lambda_{\operatorname{\prox}},
\end{split}
\end{align}
with 
\begin{eqnarray}\nonumber
\overline{\Lambda}_{\operatorname{\prox}}^2&:=&16C^2\ML  C_{b,\sigma}^2 \e^{2\underline\nu^{-2}\C(2A+A^2)}(1+\sup_{x\in\R}|x|^{2\eta}\rho_b(x))\\\label{def:lprox}
&&\hspace*{3em}+ 4C^2\max\{2^{2\eta},2\}\bigg(2K^2\ML (1+\sup_{x\in\R} |x|^{2\eta}\rho_b(x))\\\nonumber
&&\hspace*{3em}+\underline\nu^{-2}\sup_{x\geq 0}\exp(-4\gamma x)x^{2\eta}+\underline\nu^{-2}\bigg),
\end{eqnarray}
for some constant $K=K(\C,A,\gamma,\overline\nu,\underline\nu)$.
For the particular case $b_0=b$, the representation satisfies, for any $p\geq 1$ and any $f\in\FF\cup\overline{\FF}$,
\begin{align}\tag{$\mathbf{II}$}
\begin{split}
\left(\E_b\left[|\Ma_t^{f}|^p\right]\right)\p&\leq\ p\sqrt t\overline{\Gamma}_{\operatorname{\prox}} \|f\|_{L^2(\lebesgue)} \sqrt 2\overline\nu\bdg (1+\mathcal S+\mom)^{1/2},\\
\left(\E_b\left[\left(\sup_{f\in\FF}|\rd^{f}_t|\right)^p\right]\right)\p &\leq\ p \Gamma_{\operatorname{\prox}},
\end{split}
\end{align}
with 
\begin{align*}
\overline{\Gamma}^2_{\operatorname{\prox}}&:=8\left\{\frac{1}{4}K^2\ML^2+\frac{\underline\nu^{-2}\mathbb C}{2}\left(1 + \sup_{x\geq 0}\exp(-4\gamma x)x \right)+\ML^2 C^2_{b,\sigma} \e^{2\underline\nu^{-2}\C(2A+A^2)}\right\},\\
\Gamma_{\prox}&:=4U\mom\left(2K\left(2\overline\nu^2\ML  + \mathbb C(1+A)\right) + 2C_{b,\sigma}\e^{\underline\nu^{-2}\C(2A+A^2)}\left(A \mathbb C(1+A)\ML +\frac{\overline\nu^2\ML}{2}\right)\hspace*{-0.1em}+\hspace*{-0.1em}1\hspace*{-0,1em}\right)\hspace*{-0.1em}.
 \end{align*}
\end{proposition}

\medskip

\begin{remark}
The above result should be read carefully.
We consider an arbitrary continuous function $b_0$, not necessarily of compact support, satisfying some polynomial growth condition.
Our interest is in bounding the $p$-th moments of the empirical process $(\G_t^{b_0}(f))_{f\in\FF}$, indexed by the functions $\FF\ni f\colon \R\to\R$.
Neglecting constants, the first approach to analysing the moments of the martingale and of the remainder term shows that, for any $p\geq1$,
\begin{equation}\label{app1}
\frac{1}{\sqrt t}\|\Ma_t^f\|_p\ \lesssim\ p^{\eta+\frac{1}{2}}\sqrt{\mathcal S}\|f\|_{L^2(\lebesgue)},\qquad 
\Big\|\sup_{f\in\FF}|\rd_t^f|\Big\|_p\ \lesssim\ p^{\eta+1}\mathcal S.
\end{equation}
Specifying to the case $b_0=b$, one can exploit the basic relation $(\sigma^2\rho_b)'=2\rho_bb$. 
One then obtains bounds of the order
\begin{equation}\label{app2}
\frac{1}{\sqrt t}\|\Ma_t^f\|_p\ \lesssim\ p\|f\|_{L^2(\lebesgue)},\qquad\quad \Big\|\sup_{f\in\FF}|\rd_t^f|\Big\|_p\ \lesssim\ p.
\end{equation}
Regarding the exponent of $p$, \eqref{app2} is superior to the bound implied by \eqref{app1} for the specific case $\eta=1$ (which corresponds to the standard at-most-linear-growth assumption on the drift term). 
However, it will be seen below that it might be advantageous to choose the upper bound \eqref{app1} with $\eta=1$ for the martingale part.
Note that this bound provides the factor $\sqrt{\mathcal S}$.
In a number of statistical applications (e.g., the procedure that we have in mind), the support of the functions $f$ from the class $\FF$ vanishes.
Consequently, the contribution of the factor $\sqrt \mathcal S$ is more beneficial than the improvement in the tail behaviour implied by \eqref{app2}.
\end{remark}

\subsection{Uniform concentration of empirical processes and stochastic integrals}\label{subsec:4.3}
Consider some diffusion process $\X$ as introduced in Definition \ref{def:B} with invariant measure $\mu_b$, and let us briefly recall our previous outcomes. 
Proposition \ref{prop:mom} gives both a martingale approximation of the empirical process
\[\G_t^{b}(f)\ =\ \frac{1}{\sqrt t}\int_0^t \left(f(X_u)b(X_u)-\E_b[f(X_0)b(X_0)]\right)\d u\]
and bounds on the $p$-th moments of its martingale and remainder term.
Theorem \ref{theo:1} allows to translate these bounds into bounds on $\|\G_t^{b}\|_\FF^p$, $p\geq 1$, the supremum taken over entire function classes $\FF$. 
Combining both results, we obtain the following

\medskip

\begin{theorem}\label{theo:bernsup}
Let $\X$ be as in Definition \ref{def:B}.
Suppose that $\FF$ is a class of continuous functions fulfilling Assumptions \ref{ann:1} and \ref{ann:3}, and set 
$\Lambda:=\lambda \e (\max(4\mathbb C(1+\mom),\overline\nu^2)+\bdg\overline\nu)$,
where $\lambda>1$ is chosen such that $\max\{\mathcal S,\e \mom\}<\lambda \e (\max(4\mathbb C(1+\mom),\overline\nu^2)+\bdg\overline\nu)$.
Then, for any $p\geq 1$,
\begin{equation}\tag{$\mathbf{I}$}
\sup_{b\in\Sigma(\C,A,\gamma,\sigma)}\left(\E_b\left[\|\G^1_t\|_{\mathcal{F}}^p\right]\right)^{\p}
\ \leq\  \Phi_t(p),
\end{equation}
for
\[\Phi_t(u)\ :=\ \V\sqrt{\mathcal S} 
\left\{12C_2\e\ps_1\sqrt{v\log\left(\frac{\AA}{\V}\sqrt{\mathcal S + u\Lambda t}\right)} + 6\ps_1 \sqrt{2u}\right\}+2\mathcal{S}\frac{\ps_2u}{\sqrt t}+ \sqrt t C\U \e^{-\frac{\Lambda t}{2\e\mom}},\]
with $\ps_1:=\sqrt 2\overline\nu\bdg\overline{\Lambda}_{\prox}$ and $\ps_2:=4\max\{\mom,1\}\U \overline{\Lambda}_{\prox}$.
Furthermore,
\begin{equation}\tag{$\mathbf{II}$}
\sup_{b\in\Sigma(\C,A,\gamma,\sigma)}\left(\E_b\left[\|\G^{b}_t\|_{\mathcal{F}}^p\right]\right)^{\p}
\ \leq\ \Ph_t(p),
\end{equation}
where 
\begin{align}\label{def:Ph}
\begin{split}
\Ph_t(u)&:=\V\sqrt{\mathcal S}\Bigg\{3 C_{\frac{2}{3}} \e\ps_1^{b}\left(2\left(v\log \left(\frac{\AA}{\V}\sqrt{\mathcal S + u\Lambda t}\right)\right)^{\frac{3}{2}} +6v^{\frac{3}{2}}\sqrt{\log\left(\frac{\AA}{\V}\sqrt{\mathcal S + u\Lambda t}\right)}\right)\Bigg.\\
&\hspace*{9em} \Bigg.+ 6\ps^b_1 (2u)^{\frac{3}{2}}\Bigg\}+\frac{2\ps^b_2u}{\sqrt t}+ \sqrt t C\U (1+2\mom)\e^{-\frac{\Lambda t}{2\e\mom}}
\end{split}
\end{align}
and $\ps^b_1:=2^{\frac{3}{2}}\overline\nu\bdg\overline{\Lambda}_{\prox} \left(1+\mom\right)$, $\ps^b_2:=\Gamma_{\prox}$.
\end{theorem}

Our interest finally is in formulating exponential inequalities for the process
\begin{equation}\label{def:H}
\H_t(f)=\sqrt t\left(\frac{1}{t}\int_0^t f(X_s)\d X_s-\int(fb)\d\mu_b\right),\quad f\in\FF.
\end{equation}
At this point, we can apply several of our previous findings for proving one first uniform moment bound for the general stochastic integral process $(\H_t(f))_{f\in\FF}$.

\medskip

\begin{proposition}\label{theo:mart}
Grant the assumptions of Theorem \ref{theo:bernsup}. 
Then, there exists a positive constant $\mathbb L$ (depending only on $\bdg,\C,\mom,\Lambda, \U, A, \gamma, v, \AA, \underline\nu, \overline\nu$) such that, for any $p,t\geq 1$,
\[\sup_{b\in\Sigma(\C,A,\gamma,\sigma)}\left(\E_b\left[\|\H_t\|_\FF^p\right]\right)\p
\ \leq\ \mathbb L\left(\V\left(1+ \log\left(\frac{1}{\V}\right)+\log t +  p\right)+ \frac{p}{\sqrt t}+\sqrt t \e^{-\frac{\Lambda t}{2\e\mom}}\right).\]
\end{proposition}


\bigskip

\section{Statistical applications}\label{sec:tp}
This section considers the basic question of density estimation in supremum-norm which, from a general statistical point of view, is of immense theoretical and practical interest. 
Let us assume that a continuous record of observations $X^t:=(X_s)_{0\leq s\leq t}$ of a diffusion process as introduced in Definition \ref{def:B} is available, and we aim at nonparametric estimation of the associated invariant density $\rho_b$.
Given some smooth kernel function $K\colon \R\to\R$ with compact support, define the standard kernel estimator $\rho_{t,K}(h)$ according to \eqref{est:dens}.
For our statistical analysis which targets results concerning the risk in $\sup$-norm loss, i.e., the behaviour of the maximal error $\|\rho_{t,K}(h)-\rho_b\|_\infty$, we impose some regularity on $b$ and $\rho_b$. 
To be more precise, we look at H\"older classes defined as follows.

\medskip

\begin{definition}\label{def:Hölder}
Given $\beta,\mathcal{L}>0$, denote by $\mathcal{H}_{\R}(\beta,\mathcal{L})$ the \emph{H\"older class} (on $\R$) as the set of all functions $f\colon\R\to\R$ which are $l:=\lfloor \beta \rfloor$-times differentiable and for which
\begin{eqnarray*}
\|f^{(k)}\|_{\infty}&\leq&\mathcal{L}  \qquad\qquad\forall\, k=0,1,...,l,\\
\|f^{(l)}(\cdot+t)-f^{(l)}(\cdot)\|_{\infty}&\leq &\mathcal{L}|t|^{\beta- l}  \qquad \forall\, t\in\R.
\end{eqnarray*}
Set $\Sigma(\beta,\mathcal{L})\ :=\ \left\{b\in\Sigma(\C,A,\gamma,\sigma)\colon\ \rho_b\in\mathcal{H}_\R(\beta,\mathcal{L})\right\}$. 
Here, $\lfloor \beta \rfloor$ denotes the greatest integer strictly smaller than $\beta$.
\end{definition}

Considering the class of drift coefficients $\Sigma(\beta,\mathcal{L})$, we use kernel functions satisfying the following assumptions,
\begin{equation}\label{kernel}
\begin{array}{r@{}l}
&{}\bullet\quad K:\R\rightarrow \R \text{ is Lipschitz continuous and symmetric},\\ [3pt]
&{}\bullet\quad\supp(K)\subset [-1/2,1/2],\\[3pt]
&{}\bullet\quad K\text{ is of order }\lfloor\beta \rfloor.
\end{array}
\end{equation}

\smallskip

\begin{corollary}[Concentration of the kernel invariant density estimator]\label{prop:csi}
Let $\X$ be a diffusion as in Definition \ref{def:B} with $b\in\Sigma(\beta,\mathcal L)$, for some $\beta,\mathcal L>0$, and let $K$ be a kernel function fulfilling \eqref{kernel}.
Given some positive bandwidth $h$, define the estimator $\rho_{t,K}(h)$ according to \eqref{est:dens}. 
Then, there exist positive constants $\nu_1,\nu_2$ (not depending on $p$) such that, for any $p\geq 1$, $t>0$,
\begin{eqnarray*} 
\sup_{b\in\Sigma(\beta,\mathcal L)}\left(\E_b\left[\left\|\rho_{t,K}(h) - \rho_b\right\|^p_\infty\right]\right)\p
&\leq& \frac{\nu_1}{\sqrt t}\left\{1+\sqrt{\log\left(\frac{1}{\sqrt h}\right)} + \sqrt{\log(pt)}+\sqrt p\right\}\\
&&\hspace*{2.5em}+\frac{\nu_2p}{t}+\frac{1}{h}\e^{-\frac{\Lambda t}{2\e\mom}} +h^{\beta }\frac{\mathcal L}{\lfloor \beta\rfloor!}\int |u^{\beta }K(u)|\d u.
\end{eqnarray*}
\end{corollary}
\begin{proof}
We want to apply Theorem \ref{theo:bernsup} to the class 
\[\mathcal F:= \left\{K\left(\frac{x-\cdot}{h}\right):\,x\in\mathbb Q\right\}.\]
For doing so, note that $\sup_{f\in\mathcal F}\|f\|_\infty\leq \|K\|_\infty$, $\sup_{f\in\mathcal F}\lebesgue(\supp(f))\leq h$ and 
\begin{eqnarray*}
\left\|K\left(\frac{x-\cdot}{h}\right)\right\|^2_{L^2(\lebesgue)}&=& \int K^2\left(\frac{x-y}{h}\right)\d y\ =\ h \int K^2(z)\d z\ \leq\ h\|K \|^2_{L^2(\lebesgue)}.
\end{eqnarray*}
Setting $\mathcal S := h\ \max\{\|K \|^2_{L^2(\lebesgue)},1\}$, $\V:= \sqrt h\|K \|_{L^2(\lebesgue)}$ and taking into account Lemma \ref{lem:entgi}, $\FF$ is seen to satisfy Assumptions \ref{ann:1} and \ref{ann:3}. Thus, Theorem \ref{theo:bernsup} with $b_0=1$ is applicable. 
In particular, there exist positive constants $\nu_1$ and $\nu_2$ such that, for any $p\geq 1$,
\begin{eqnarray*} 
&&\left(\E_b\left[\left\|\frac{1}{th}\int_0^t K\left(\frac{x-X_u}{h}\right)\d u-\E_b\left[\frac{1}{h}K\left(\frac{x-X_0}{h}\right)\right]\right\|^p_\infty\right]\right)\p\\
&&\hspace*{3em}=\ \frac{1}{\sqrt t h}\left(\E_b\left[\left\|\sqrt t\left\{\frac{1}{t}\int_0^t K\left(\frac{x-X_u}{h}\right)\d u - \E_b\left[K\left(\frac{x-X_0}{h}\right) \right]\right\}\right\|^p_\infty\right]\right)\p\\
&&\hspace*{3em}\leq\ \frac{1}{\sqrt t h}\ h \nu_1\left\{1 +\sqrt{\log\left(\frac{1}{\sqrt h}\right)} + \sqrt{\log(pt)}+\sqrt{p}\right\}+\frac{\nu_2p}{t}+\frac{1}{h}\exp\left(-\frac{\Lambda t}{2\e\mom}\right).
\end{eqnarray*}
For the bias, we obtain
\begin{eqnarray*}
\left|\frac{1}{th}\E_b\left[\int_0^tK\left(\frac{x-X_u}{h}\right)\d u\right] - \rho_b(x)\right|
&=& \left| \frac{1}{h} \int K\left(\frac{x-y}{h}\right)\left(\rho_b(y)-\rho_b(x)\right)\d y \right|\\
&\leq&h^{\beta }\frac{\mathcal L}{\lfloor \beta\rfloor!}\int |u^{\beta }K(u)|\d u.
\end{eqnarray*}
Combining the above estimates, the assertion follows.
\end{proof}

Recall that $\L_t^x(\X)$ denotes diffusion local time and that $\rho_t^\circ(x)=t^{-1}\L_t^x(\X)\sigma^{-2}(x)$ is the associated local time estimator of the value of the invariant density $\rho_b(x)$ of $X$.
We now turn to deriving an exponential inequality for the tail probabilities of $\sqrt t\|\rho_{t,K}(h)-\rho_t^\circ\|_\infty$ which holds under rather mild assumptions on the diffusion $\X$ and the bandwidth $h$. 
It can be interpreted as some analogue of Theorem 1 in \cite{gini09} where the authors investigate the maximum deviation between the classical empirical distribution function (based on i.i.d.~observations) and the distribution function obtained from kernel smoothing.
The proof of Theorem \ref{theo:cath} substantially relies on Proposition \ref{theo:mart}. 
Throughout the sequel, we restrict to a constant diffusion coefficient $\sigma^2\equiv 1$ in order to ease the exposition. 
Still, our methods are suitable to treat more general diffusion coefficients under H\"older-smoothness conditions on $\sigma^2$ that correspond to the conditions on the invariant density.

\medskip

\begin{theorem}\label{theo:cath}
Let $\X$ be a diffusion as in Definition \ref{def:B} with $\sigma^2\equiv 1$ and $b\in\Sigma(\beta,\mathcal L)$, for some $\beta,\mathcal L>0$.
Consider some kernel function $K$ fulfilling \eqref{kernel} and $h=h_t\in(0,1)$ such that $h_t\geq t^{-1}$. 
Then, there exist positive constants $\VV$, $\Lambda$, $\Lambda_0$ and $\co$ such that, for all
\begin{align*}
&\lambda\ \geq\ &\hspace*{-4em}8\Lambda_0\bigg[ \sqrt h \VV\e\mathbb L\left\{1+ \log\left(\frac{1}{\sqrt h \VV}\right)+\log t\right\}+\e\co\sqrt t \exp\left(-\frac{\Lambda t}{2\e\mom}\right)\\
&&\hspace*{2em}+\sqrt t h^{\beta}\frac{\mathcal L}{2\lfloor \beta\rfloor!}\int |K(v)v^{\beta }|\d v\bigg]
\end{align*}
and any $t>1$,
\begin{equation}\label{eq:theo:1}
\sup_{b\in\Sigma(\beta,\LL)} \P_b\left(\sqrt t\left\|\rho_{t,K}(h)-\rho_t^\circ\right\|_\infty>\lambda\right)\ \leq \ \exp\left(-\frac{\Lambda_1\lambda}{\sqrt h}\right).
\end{equation}
\end{theorem}

One first application of Theorem \ref{theo:cath} concerns the derivation of an upper bound on the $\sup$-norm risk of the diffusion local time estimator.
In fact, it allows to prove the following Corollary which shows that, concerning invariant density estimation, the procedures based on the kernel density and the local time estimator, respectively, are of equal quality in terms of $\sup$-norm rates of convergence.

\medskip

\begin{corollary}\label{cor:centlt}
Let $\X$ be a diffusion as in Definition \ref{def:B} with $\sigma^2\equiv 1$.
Then, there is a positive constant $\zeta$ such that, for any $p,t\geq 1$,
\begin{equation}\label{diff:tail1}
\sup_{b\in\Sigma(\C,A,\gamma,1)}
\left(\E_b\left[\Big\|\frac{\L_t^\bullet(\X)}{t}-\rho_b\Big\|_\infty^p\right]\right)\p
\ \leq\ \zeta\left(\frac{p}{t}+ \frac{1 + \sqrt p+\sqrt{\log t}}{\sqrt t}+t\e^{-\frac{\Lambda t}{2\e\mom}}\right).
\end{equation}
In addition, for any $u\geq 1$,
\begin{equation}\label{diff:tail2}
\P_b\left(\left\|\L_t^\bullet(\X)-t\rho_b\right\|_\infty\geq\e\zeta\left(\sqrt t\left(1+\sqrt{\log(ut)}+\sqrt u\right)+u+t^2 \e^{-\frac{\Lambda t}{2\e\mom}}\right)\right)\ \leq\ \e^{-u}.
\end{equation}
\end{corollary}

\medskip

\begin{remark}
\begin{itemize}
\item[(a)]
As already indicated, the results yield the same $\sup$-norm convergence rate for the local time and the kernel density estimator with bandwidth $t^{-1/2}$, i.e.,
\[\sup_{b\in\Sigma(\C,A,\gamma,1)}
\left(\E_b\left[\left\|\tilde \rho_t-\rho_b\right\|_\infty^p\right]\right)\p =O\left(\left(\frac{\log t}{t}\right)^{1/2}\right),\text{ for }\tilde \rho_t\in\left\{\rho_{t,K}(t^{-1/2}), t^{-1}L^\bullet_t(X)\right\}.
\]
\item[(b)]
The explicit dependence of the minimax upper bounds in Corollary \ref{prop:csi} and Corollary \ref{cor:centlt} on $p$ is crucial for further statistical applications such as adaptive drift estimation. 
As compared to Corollary \ref{cor:centlt}, we do not have to impose additional smoothness assumptions on the drift coefficient for applying Corollary \ref{prop:csi} since $b\in\Sigma(\C,A,\gamma,1)$ implies that $b\in\Sigma(1,\mathcal L)$.
\item[(c)] Since the local time estimator is unbiased, Corollary \ref{cor:centlt} can also be interpreted as a result on the centred local time, providing a concentration inequality of the form \eqref{diff:tail2} which is of its own probabilistic interest.
\end{itemize}
\end{remark}

\smallskip

Once the result for the centred local time stated in \eqref{diff:tail1} is available, one can derive the following modified version of Proposition \ref{theo:mart}. 
In a number of concrete applications, this version can be considered as an improvement, even though we lose the subexponential behaviour. 
This is our price for obtaining a better upper bound in terms of the size $\mathcal S$ of the support of the functions from the function class $\mathcal F$. 
In our statistical application, the support is of size $h_t$ with $h_t\downarrow 0$ as $t\to\infty$. 
Therefore, gaining another $\sqrt{\mathcal S}$ is more beneficial than the subexponential behaviour.
Recall the definition of $\H_t$ in \eqref{def:H}.

\medskip

\begin{theorem}\label{improved_version}
Let $\X$ be a diffusion as in Definition \ref{def:B} with $\sigma^2\equiv 1$, and grant the assumptions of Theorem \ref{theo:bernsup}. 
Then, for any $p,t\geq1$, there exist constants $\tilde\co$ and $\tilde\co_0$ such that
\begin{equation}\label{eq:mo18}
\sup_{b\in \Sigma(\C,A,\gamma,1)}\left(\E_b\left[\|\H_t\|_\FF^p\right]\right)\p \ \leq\  \tilde\Psi_t(p),
\end{equation}
where
\begin{align*}
\tilde\Psi_t(p)&:=\ \tilde\co\Bigg\{ 
\V\sqrt{\mathcal S}\Big\{\left(\log \left(\frac{\AA}{\V}\sqrt{\mathcal S + p\Lambda t}\right)\right)^{3/2} + \left(\log\left(\frac{\AA}{\V}\sqrt{\mathcal S + p\Lambda t}\right)\right)^{1/2} + p^{3/2}\Big\}+\frac{p}{\sqrt t}\\
&\hspace*{3em}+ \sqrt t \exp\left(-\tilde\co_0 t \right)
+ \V\left(\log\left(\frac{\AA}{\V}\sqrt{\mathcal S + p\Lambda t}\right)\right)^{1/2}\\
&\hspace{3em}+\frac{\V}{t^{1/4}}\left(1+ \log\left(\frac{\AA}{\V}\sqrt{\mathcal S + p\Lambda t}\right)\right) + \V \left\{\sqrt p+ \frac{p}{t^{1/4}}\right\}\Bigg\}.
\end{align*}
\end{theorem}

\medskip

\begin{remark}
As before, it is straightforward to translate the moment bound \eqref{eq:mo18} into a corresponding upper tail bound by means of Lemma \ref{moment_tail_lemma}.
The effectiveness of the obtained exponential inequalities is reinforced in \cite{cacs18} where we investigate the question of adaptive drift estimation. 
In this respect, Theorem \ref{improved_version} on stochastic integrals will be a crucial device.
\end{remark}

\bigskip
\noindent
\textbf{Acknowledgment} 
The authors thank Richard Nickl for his interest in this work and very helpful discussions.
Financial support from the Deutsche Forschungsgemeinschaft via RTG 1953 is gratefully acknowledged.

\bigskip

\begin{appendix}

\section{Basic auxiliary results}\label{app:A}
We start with proving two auxiliary results which are frequently used in our analysis.

\medskip

\begin{lemma} \label{X_0_Mom} 
Let $X$ be as in Definition \ref{def:B}. 
Then, there is a positive constant $\mom$, depending only on $\C, A,\gamma,\underline\nu,\overline\nu$, such that 
\[\sup_{b\in\Sigma(\C,A,\gamma,\sigma)}\|X_0\|_p\ = \ \sup_{b\in\Sigma(\C,A,\gamma,\sigma)} \left(\E_b\left[|X_0|^p\right]\right)\p\ \leq\ \mom p,\quad p\geq 1.\]
\end{lemma}
\begin{proof}
Note that
\begin{eqnarray*}
&&\sup_{b\in \Sigma(\C,A,\gamma,\sigma)}\E_b\left[|X_0|^p\right]=\sup_{b\in\Sigma(\C,A,\gamma,\sigma)}\int |x|^p\rho_b(x)\d x\\
&&\hspace*{3em}\leq 2A^{p+1}\ML  + \int_A^\infty x^p\exp(-2\gamma(x-A))\d x\ \left(\rho_b(A)+\rho_b(-A)\right)\underline\nu^{-2}\overline\nu^2\\
&&\hspace*{3em}\leq 2A^{p+1}\ML  + \underline\nu^{-2}\overline\nu^2(\rho_b(A)+\rho_b(-A))\left(2^{p-1}\frac{A^p}{2\gamma} +\frac{2^{p-1}}{(2\gamma)^{p+1}}\int_0^\infty x^p\e^{-x}\d x\right)\\
&&\hspace*{3em}= 2A^{p+1}\ML  + 2\ML\underline\nu^{-2}\overline\nu^2\left(2^{p-1}\frac{A^p}{2\gamma} +\frac{2^{p-1}}{(2\gamma)^{p+1}}\Gamma(p+1)\right).
\end{eqnarray*}
Due to the formula of Stirling, we have
\[\left(\Gamma(p+1)\right)^{\frac{1}{p}}\leq \sqrt{2\pi}\e (p+1)^{1+1/p}\leq  \sqrt{2\pi}\e (p+1)\tilde c\leq \sqrt{2\pi} 2\e\tilde c p  \]
for a constant $\tilde c$ such that $\sup_{p\geq 1}(p+1)^{1/p}\leq \tilde c$. This gives the assertion.
\end{proof}

\medskip

\begin{lemma}\label{moment_tail_lemma}
Let $X$ be a real-valued random variable satisfying, $\forall p\geq 1$, $(\E\left[|X|^p\right])^{\frac{1}{p}}\leq f(p)$, for some function $f\colon(0,\infty)\rightarrow(0,\infty)$. 
Then,
\begin{equation}\label{eq_tail_lemma}
\P\left(|X|\geq \e f(u)\right)\ \leq\ \exp\left(-u\right),\quad u\geq 1.
\end{equation}
\end{lemma}
\begin{proof}
Fix $u\geq 1$. Then, for any $p \geq 1$,
\begin{eqnarray*}
\P\left(|X|\geq \e f(u)\right)\leq \frac{\E\left[|X|^p\right]}{\e^pf^p(u)}\ \leq\ \frac{f^p(p)}{\e^pf^p(u)}.
\end{eqnarray*}
Setting $p:=u$, we obtain \eqref{eq_tail_lemma}.
\end{proof}

\medskip

One central ingredient for the proof of our concentration inequalities are generic chaining results which go back to Talagrand (cf.~\cite{tala96} and \cite{tala14}). 
We state a version of the results in \cite{dirk15} here which is adjusted to our needs. 
In particular, we bound the abstract truncated $\gamma$-functionals appearing in \cite{dirk15} by entropy integrals.

\medskip

\begin{proposition}[cf.~Theorem 3.2 \& 3.5 in \cite{dirk15}]\label{thm:dirk}
Consider a real-valued process $(X_f)_{f\in\mathcal F}$, defined on a semi-metric space $(\FF,d)$.
\begin{itemize}
\item[$\operatorname{(a)}$] If there exists some $\alpha\in(0,\infty)$ such that
\begin{equation}\label{ineq:a}
\mathbb P\left(|X_f-X_g|\geq u d(f,g)\right)\ \leq\ 2\exp\left(-u^{\alpha}\right)\quad\forall  f,g\in\mathcal F,\ u\geq 1,
\end{equation}
then there exists some constant $C_\alpha>0$ (depending only on $\alpha$) such that, for any $1\leq p<\infty$,
\begin{equation}\label{eq:101}
\left(\mathbb E\left[\sup_{f\in\mathcal F}|X_f|^p\right]\right)^{\frac{1}{p}}\ \leq\ C_\alpha \int_0^\infty \left(\log N(u,\mathcal F,d)\right)^{\frac{1}{\alpha}}\d u + 2\sup_{f\in\mathcal F}\left(\mathbb E\left[|X_f|^p\right]\right)^{\frac{1}{p}}.
\end{equation}
\item[$\operatorname{(b)}$] If there exist semi-metrics $d_1,d_2$ on $\FF$ such that
\[\P\left(|X_f-X_g|\geq ud_1(f,g)+\sqrt u d_2(f,g)\right)\ \leq\ 2\e^{-u}\quad \forall f,g\in\FF,\ u\geq 1,
\]
then there exist positive constants $\tilde C_1,\tilde C_2$ such that, for any $1\leq p<\infty$,
\begin{eqnarray}\label{eq:102}
\left(\mathbb E\left[\sup_{f\in\mathcal F}|X_f|^p\right]\right)^{\frac{1}{p}}
&\leq& \tilde C_1 \int_0^\infty \log N(u,\FF,d_1)\d u\\\nonumber
&&\hspace*{3em} + \tilde C_2 \int_0^\infty \sqrt{\log N(u,\FF,d_2)}\d u + 2\sup_{f\in\FF}\left(\mathbb E\left[|X_f|^p\right]\right)^{\frac{1}{p}}.
\end{eqnarray}
\end{itemize}
\end{proposition}

The entropy integrals appearing on the rhs of \eqref{eq:101} and \eqref{eq:102} will be controlled by means of the following lemmata.

\medskip

\begin{lemma}\label{lem:entgi}
Given some function of bounded variation $H\colon\R\to\R$ and $h>0$, let 
\[\FF\ :=\ \FF_h\ =\ \left\{H\left(\frac{x-\cdot}{h}\right)\colon x\in\R\right\}.\]
Then there exist some constants $A=A(\|H\|_{\operatorname{TV}})<\infty$ and $v\geq 2$, not depending on $h$, such that, for any probability measure $\mathbb{Q}$ on $\R$ and any $0<\ep<1$, $N(\ep,\FF_h,\|\cdot\|_{L^2(\mathbb{Q})})\ \leq\ (A/\ep)^v$.
\end{lemma}
The preceding lemma is a consequence of the more general result of Proposition 3.6.12 in \cite{gini16}.

\medskip

\begin{lemma}\label{EntropyIntegralLebesgueMetric}
Grant the conditions of Theorem \ref{theo:1} and Assumption \ref{ann:3}, and define the function classes $\FF_k$ according to \eqref{ikfk}. 
Then, for all $k\in\N_0$ and any constant $\Gamma\geq 1$,
\begin{eqnarray*}
\int_0^\infty \log N(u,\mathcal F_k, \Gamma \|\cdot\|_{L^2(\lebesgue)})\d u&\leq& 2v\V\Gamma\left(1+ \log\left(\frac{\AA}{\V}\sqrt{\mathcal S + p\Lambda(t)}\right)\right),\\
\int_0^\infty \sqrt{\log N(u,\mathcal F_k, \Gamma \|\cdot\|_{L^2(\lebesgue)})}\d u
&\leq& 4\V\Gamma\sqrt{v\log\left(\frac{\AA}{\V}\sqrt{\mathcal S + p\Lambda(t)}\right)}\\
\int_0^\infty \left(\log N(u,\mathcal F_k, \Gamma \|\cdot\|_{L^2(\lebesgue)})\right)^{3/2}\d u
&\leq&2\V\Gamma\left(v\log \left(\frac{\AA}{\V}\sqrt{\mathcal S + p\Lambda(t)}\right)\right)^{3/2}\\
&&\hspace*{3em} + 6v\V\Gamma\sqrt{v\log\left(\frac{\AA}{\V}\sqrt{\mathcal S + p\Lambda(t)}\right)}.
\end{eqnarray*}
\end{lemma}
\begin{proof}
Note that, for $f\in\mathcal F_k$,  
\[
	\|f\|_{L^2(\lebesgue)}\ \leq \ \|f\|_{L^2(\nu_k)} \sqrt{4\mathcal S + 4p\Lambda(t)},\quad \text{ where }\d \nu_k\ =\ \mathds{1}\{I_k \}\d\frac{\lebesgue}{\lebesgue(I_k)}.
\] 
Thus, \eqref{eq:ent} implies that
\begin{eqnarray*}
N\left(u,\mathcal F_k, \Gamma \|\cdot\|_{L^2(\lebesgue)}\right)
\ &\leq& N\left(u\left(\Gamma\sqrt{4\mathcal S + 4p\Lambda(t)}\right)^{-1},\mathcal F_k, \|\cdot\|_{L^2(\nu_k)}\right)\\
 &\leq& \left(\frac{\AA\Gamma}{u}\sqrt{4\mathcal S + 4p\Lambda t}\right)^v,
\end{eqnarray*}
if $u<2\V\Gamma\leq\Gamma\sqrt{4\mathcal S + 4p\Lambda(t)}$. Furthermore, since $\sup_{f,g\in\FF}\|f-g\|_{L^2(\lebesgue)}\leq2\V$, 
it holds that $N(u,\mathcal F_k,\Gamma \|\cdot\|_{L^2(\lebesgue)})\ =\ 1$ for $u\geq 2\V\Gamma$.
Thus, for $\alpha=1$, we can upper bound the entropy integral as follows,
\begin{align*}
\int_0^\infty \log N(u,\mathcal F_k, \Gamma \|\cdot\|_{L^2(\lebesgue)})\d u&\leq\ 
\int_0^{2\V\Gamma} v \log\left(\frac{\AA\Gamma}{u}\sqrt{4\mathcal S + 4p\Lambda(t)}\right)\d u\\
&=\ v\left[u\log\left(\frac{\AA\Gamma}{u}\sqrt{4\mathcal S + 4p\Lambda(t)}\right)\right]_0^{2\V\Gamma} + 2v\V\Gamma\\
&=\ 2v\V\Gamma\left(1+ \log\left(\frac{\AA}{\V}\sqrt{\mathcal S + p\Lambda(t)}\right)\right).
\end{align*}
For $\alpha=2$, it holds
\begin{align*}
\int_0^\infty \sqrt{\log N(u,\mathcal F_k, \Gamma \|\cdot\|_{L^2(\lebesgue)})}\d u&\leq\ 
\int_0^{2\V\Gamma} \sqrt v \sqrt{\log\left(\frac{\AA\Gamma}{u}\sqrt{4\mathcal S + 4p\Lambda(t)}\right)}\d u\\
&\leq\ \sqrt v 4\V\Gamma\left(\log\left(\frac{\AA}{\V}\sqrt{\mathcal S + p\Lambda(t)}\right)\right)^{1/2},
\end{align*}
where the last estimate is due to the fact that 
$\int_0^c\sqrt{\log(C/x)}\d x \leq 2 c \sqrt{\log(C/c)}$ for $\log(C/c)\geq 2$ (see, e.g., \cite{gini09}, p.~591).
This last condition is fulfilled in our situation since $\V\leq\sqrt{\mathcal S}$ and $\AA>\e^2$.
Finally, if $\alpha=2/3$,
\begin{align*}
\int_0^\infty \left(\log N(u,\mathcal F_k, \Gamma \|\cdot\|_{L^2(\lebesgue)})\right)^{3/2}\d u
&\leq\ v^{3/2}\int_0^{2\V\Gamma}\left(\log \left(\frac{\AA\Gamma}{u}\sqrt{4\mathcal S + 4p\Lambda(t)}\right)\right)^{3/2}\d u\\
&=\ v^{3/2}\left.u \left(\log \left(\frac{\AA\Gamma}{u}\sqrt{4\mathcal S + 4p\Lambda(t)}\right)\right)^{3/2}\right|_0^{2\V\Gamma}\\
&\qquad + v^{3/2}\int_0^{2\V\Gamma}\frac{3}{2}\left(\log \left(\frac{\AA}{u}\Gamma\sqrt{4\mathcal S + 4p\Lambda(t)}\right)\right)^{1/2}\d u\\
&\leq\ v^{3/2}2\V\Gamma\left(\log \left(\frac{\AA}{2\V}\sqrt{4\mathcal S + 4p\Lambda(t)}\right)\right)^{3/2} \\
&\qquad+ v^{3/2}\frac{3}{2}4\V\Gamma\left(\log\left(\frac{\AA}{\V}\sqrt{\mathcal S + p\Lambda(t)}\right)\right)^{1/2}.
\end{align*}
\end{proof}

\medskip

\section{Proofs for Section 2}\label{App:B}
\begin{proof}[Proof of Theorem \ref{theo:lt}]


Tanaka's formula (see Proposition 9.2 in \cite{legall16}) yields the local time representation
\begin{align*}
\L_t^a(X)& = \L_t^a(X)\cdot \mathds{1}\left\{\max_{0\leq s\leq t}|X_s|\geq |a|\right\}\\
&= 2\left((X_t-a)^--(X_0-a)^-+\int_0^t\mathds{1}\{X_s\leq a\}\d X_s\right),\hspace*{0.5em} \text{ where } x^-:=\max\left\{-x,\ 0\right\}.
\end{align*}
Since semimartingale local time is c\`adl\`ag in $a$, the $\sup$-norm actually refers to a supremum over the rationals $\mathbb Q$.
In particular, $\|L^\bullet_t(X)\|_\infty$ is measurable.
Furthermore, for any $t>0$ and $p\geq 1$,
\begin{eqnarray}\nonumber
&&\left(\E\left[\left(\sup_{a\in\mathbb{Q}}|\L_t^a(X)|\right)^p\right]\right)\p
= \left(\E\left[\left(\sup_{a\in\mathbb{Q}}\left\{|\L_t^a(X)| \cdot\mathds{1}\left\{\max_{0\leq s\leq t}|X_s|\geq |a|\right\}\right\}\right)^p\right]\right)\p\\\nonumber
&&\hspace*{3em}\leq 2\bigg(\E\bigg[\bigg(\sup_{a\in\mathbb{Q}}\bigg\{|X_t-X_0|+\int_0^t|\d V_s|\\\label{eq25}
&&\hspace*{4em}+ \mathds{1}\left\{\max_{0\leq s\leq t}|X_s|\geq |a|\right\}\Big|\int_0^t\mathds{1}\{X_s\leq a\}\d M_s\Big|\bigg\}\bigg)^p\bigg]\bigg)\p\\ \nonumber
&&\hspace*{3em}\leq 2p\phi_1(t) + 2\bigg(\E\bigg[\sup_{a\in\mathbb Q}\mathds{1}\left\{\max_{0\leq s\leq t}|X_s|\geq |a|\right\}\Big|\int_0^t\mathds{1}\{X_s\leq a\}\d M_s\Big|\bigg\}\bigg)^p\bigg]\bigg)\p,
\end{eqnarray}
where the latter inequality is due to \eqref{eq:p}. 
Recall that $\M^a_t=\int_0^t\mathds{1}\{X_s\leq a\}\d M_s$, $a\in\R$, and note again that \eqref{4.2} and \eqref{eq:p} imply that
\begin{equation}\label{eq27}
\sup_{a\in\mathbb{Q}}\left(\E\left[|\M_t^a|^p\right]\right)\p\  \leq\ \sup_{a\in\mathbb{Q}}  \sqrt 2\overline\bdg \sqrt p \left(\E\left[\langle\M_t^a\rangle^p\right]\right)^{\frac{1}{2p}}
 \leq\  \bdg\sqrt p\phi_2(t),\quad p\geq 1.
\end{equation}
This result provides an upper bound for the expression appearing on the rhs of \eqref{eq:101} in Proposition \ref{thm:dirk}. 
In order to apply this result, we still have to verify  the condition on $\M$, i.e., we have to find a suitable metric structure. 
For analysing the expression
\[\left|\int_0^t\mathds{1}\{a<X_s\leq b\}\d M_s\right|\ =\ \left|\M_t^a-\M_t^b\right|,\quad a\leq b,\]
we require an exponential inequality for the tail probability of these increments. 
We will deduce this inequality by investigating the corresponding moments. 
The derivation of the upper bounds relies heavily on the following auxiliary result.

\medskip

\begin{lemma}[cf.~Lemma 9.5 in \cite{legall16}]\label{lem:9.5}
Consider a continuous semimartingale $X$ satisfying Assumption \ref{M+V_upper_bound}, and write $X=X_0+M+V$ for its canonical decomposition. 
Let $p\geq 1$.
Then, for every $a,b\in\R$ with $a\leq b$ and every $t\geq0$, we have
\[
\E\left[\left(\int_0^t\mathds{1}\{a<X_s\leq b\}\d\langle M\rangle_s\right)^p\right]\ \leq\ 
2(16(b-a))^p\left\{\bdg^p p^{p/2}\phi_2^p(t)\ + \ \E\left[\left(\int_0^t|\d V_s|\right)^p\right]\right\}.
\]
\end{lemma}
Now, for any $a\leq b\in\R$ and $p\geq 1$, Lemma \ref{lem:9.5} and \eqref{4.2} give
\begin{eqnarray*}
&&\E\left[|\M_t^a-\M_t^b|^p\right]\\
&&\hspace*{0.5em}\leq\ \sqrt{\E\left[\bigg|\int_0^t\mathds{1}\{a<X_s\leq b\}\d M_s\bigg|^{2p}\right]}\\
&&\hspace*{0.5em}\leq\ \bdg^pp^{p/2}\sqrt{\E\left[\left(\int_0^t\mathds{1}\{a<X_s\leq b\}\d \langle M\rangle_s\right)^{p}\right]}\\
&&\hspace*{0.5em}\leq\ 
\text{min}\bigg\{2(16(b-a))^{p/2} \left(\bdg^{3p/2}p^{3p/4}\phi_2^{p/2}(t)+\bdg^pp^{p/2}
\sqrt{\E\left[\left(\int_0^t|\d V_s|\right)^{p}\right]}\right),\ \bdg^pp^{p/2}\phi_2^p(t)\bigg\}\\
&&\hspace*{0.5em}\leq\
\text{min}\left\{2(16(b-a))^{p/2} \left(\bdg^{3p/2}p^{3p/4}\phi_2^{p/2}(t)+\bdg^pp^{p}\phi_1^{p/2}(t)\right),\ \bdg^pp^{p/2}\phi_2^p(t)\right\}
\end{eqnarray*}
such that
\[
\left(\E\left[|\M_t^a-\M_t^b|^p\right]\right)^{1/p}\ \leq\ 
p \min\left\{8\sqrt{|a-b|},\ 1\right\} \left(\bdg \sqrt{\phi_1(t)}+\bdg^{3/2}\sqrt{\phi_2(t)}\right).
\]
Consequently (cf.~Lemma \ref{moment_tail_lemma}), the process $(\M_t^a)_{a\in\R}$ exhibits a subexponential tail behaviour wrt the metric $d_1$, defined as
\[d_1(a,b)\ :=\ \min\left\{8\sqrt{|a-b|},\ 1\right\}\ \e\left(\bdg \sqrt{\phi_1(t)}+\bdg^{3/2}\sqrt{\phi_2(t)}\right),\quad a,b\in\R,\]
that is,
\begin{equation}\label{Mtailbehavior}
\P\left(|\M_t^a-\M_t^b|\geq d_1(a,b)u\right)\ \leq\ \exp(-u),\quad u\geq 1.
\end{equation}
At this point, we would like to apply Proposition \ref{thm:dirk}. 
Since the entire real line $\R$ cannot be covered with a finite number of $d_1$-balls, we will use the maximal inequality \eqref{maxineqX} in order to apply the chaining procedure \emph{locally} on finite intervals.
For setting up the localisation procedure, fix $p_0\geq 1$, and introduce the intervals
\begin{eqnarray*}
A_0^{p_0}&:=& \left[-p_0\Lambda(t),\ p_0\Lambda(t)\right],\\
A_k^{p_0}&:=& \big[-(k+1)p_0\Lambda(t),\ -kp_0\Lambda(t)\big)\cup\big(kp_0\Lambda(t),\ (k+1)p_0\Lambda(t)\big],\quad k\in\N,
\end{eqnarray*}
with $\Lambda(t)\ \equiv\ \e\left(\phi_1(t)+\bdg \phi_2(t)\right)$.

\medskip

\begin{lemma}\label{lem4}
Define $d_1(t):= \e\left(\bdg \sqrt{\phi_1(t)}+\bdg^{3/2}\sqrt{\phi_2(t)}\right)$. For the $d_1$-entropy integrals of $A_k^{p_0}$, $k\in\N_0$, the following bound (not depending on $k$) holds true,
\[
\int_0^\infty \log N(u,A_k^{p_0},d_1)\d u\ \leq\ d_1(t)\left(4+2\log\left(8\sqrt{2p_0\Lambda(t)}\right)\right).
\]
\end{lemma}
\begin{proof}
Fix $k\in\N_0$. Given any $\ep\in(0,d_1(t))$, a decomposition of the sets $A_k^{p_0}$ into intervals of length $(\ep/(8d_1(t)))^2$ gives
$N(\ep,A_k,d_1)\ \leq\ 2p_0\Lambda(t)\left(\ep/(8d_1(t))\right)^{-2}+1$.
Moreover, it is clear that $N(\ep,A_k^{p_0},d_1)=1$ for any $\ep\geq d_1(t)$. 
For the entropy integral, we thus obtain the estimate
\begin{eqnarray*}
\int_0^\infty\log N(u,A_k^{p_0},d_1)\d u&\leq& 2\int_0^{d_1(t)}\log\left(\sqrt{2p_0\Lambda(t)}\ \frac{8d_1(t)}{u}+1\right)\d u\\
&\leq& d_1(t)\left(4+2\log\left(8\sqrt{2p_0\Lambda(t)}\right)\right).
\end{eqnarray*} 
\end{proof}
Taking into account Proposition \ref{thm:dirk}, \eqref{eq27}, \eqref{Mtailbehavior} and the previous lemma allow to deduce the local result. 
For every $k\in\N_0$, $p_0,p\geq 1$, we obtain
\[
\bigg(\E\bigg[\sup_{a\in A_k^{p_0}}\Big|\int_0^t\mathds{1}\{X_s\leq a\}\d M_s\Big|^p\bigg]\bigg)\p\
\leq\ C_1 d_1(t)\left(4+2\log\left(8\sqrt{2p_0\Lambda(t)}\right)\right) +  2\bdg\sqrt p\phi_2(t).
\]
Exploiting the fact that the probability that the support of the local time intersects with the sets $A_k^{p_0}$ vanishes, we can extend this result to the whole real line. 
Precisely, we use that, for any $k\in\N$, setting $u\equiv kp_0$, 
\[\e\left(u\phi_1(t)+ \bdg\sqrt u\phi_2(t)\right)\ 
\leq\ kp_0\e\left(\phi_1(t)+\bdg\phi_2(t)\right)\ =\ kp_0\Lambda(t),\]
and consequently, according to Lemma \ref{maximal_inequality_X},
\[
\P\left(\max_{0\leq s\leq t}|X_s|>kp_0\Lambda(t)\right)\ \leq\ \exp\left(-kp_0\right).
\]
Moreover, it holds
\[\sum_{k=1}^\infty\exp\left(-\frac{k}{2}\right)=
\sum_{k=0}^\infty\exp\left(-\frac{(k+1)}{2}\right)\ \leq\ \int_0^\infty\exp\left(-\frac{x}{2}\right)\d x\ =\ 2\int_0^\infty\e^{-y}\d y\ =\ 2.\]
Coming back to the decomposition \eqref{eq25}, we finish the proof by noting that, for any $p_0\geq 1$,
\begin{eqnarray*}
&&\left(\E\left[\left(\sup_{a\in\mathbb{Q}}\left\{|\M_t^a|\cdot \mathds{1}\left\{\max_{0\leq s\leq t}|X_s|\geq |a|\right\}\right\}\right)^{p_0}\right]\right)^{\frac{1}{p_0}}\\
&&\hspace*{3em}\leq\ \left(\E\left[\left(\sup_{a\in A_0^{p_0}}|\M_t^a|\right)^{p_0}\right]\right)^{\frac{1}{p_0}}\\
&&\hspace*{6em} + \sum_{k=1}^\infty\left(\E\left[\left(\sup_{a\in A_k^{p_0}}
\left\{|\M_t^a|\cdot\mathds{1}\left\{\max_{0\leq s\leq t}|X_s|\geq |kp_0\Lambda(t)|\right\}\right\}\right)^{p_0}\right]\right)^{\frac{1}{p_0}}\\
&&\hspace*{3em}\leq\ \left(\E\left[\left(\sup_{a\in A_0^{p_0}}|\M_t^a|\right)^{p_0}\right]\right)^{\frac{1}{p_0}}\\
&&\hspace*{6em} + \sum_{k=1}^\infty\Bigg[\left(\E\left[\sup_{a\in A_k^{p_0}}
|\M_t^a|^{2p_0}\right]\right)^{\frac{1}{2p_0}}\P\left(\max_{0\leq s\leq t}|X_s|\geq |kp_0\Lambda(t)|\right)^{\frac{1}{2p_0}}\Bigg]\\
&&\hspace*{3em}\leq\  C_1 d_1(t)\left(4+2\log\left(8\sqrt{2p_0\Lambda(t)}\right)\right) +  2\bdg\sqrt p_0\phi_2(t)\\
&&\hspace*{6em} +  \sum_{k=1}^\infty \exp\left(-\frac{k}{2}\right)\Bigg[C_1 d_1(t)\left(4+2\log\left(8\sqrt{2p_0\Lambda(t)}\right)\right) +  2\bdg\sqrt{2p_0}\phi_2(t)\Bigg]\\
&&\hspace*{3em}\leq\ 3C_1 d_1(t)\left(4+2\log\left(8\sqrt{2p_0\Lambda(t)}\right)\right) +  10\bdg\sqrt p_0\phi_2(t).
\end{eqnarray*}
Summing up, we can conclude that, for any $p_0\geq1$,
\begin{align*}
&\left(\E\left[\left(\sup_{a\in\mathbb{Q}}|\L_t^a(X)|\right)^{p_0}\right]\right)^{\frac{1}{p_0}}\\
&\hspace*{3em}\leq 2p_0\phi_1(t)+3C_1 d_1(t)\left(4+2\log\left(8\sqrt{2p_0\Lambda(t)}\right)\right) +  10\bdg\sqrt p_0\phi_2(t)\\
&\hspace*{3em}\leq \kappa\left(p_0\phi_1(t) + \left(\sqrt{\phi_1(t)} + \sqrt{\phi_2(t)}\right)\log(2p_0\Lambda(t)) + \sqrt{p_0}\phi_2(t)\right),
\end{align*}
for a positive constant $\kappa$ depending only on $\bdg, C_1$.
\end{proof}

\medskip

\section{Proofs for Section \ref{sec:cmei}}\label{sec:B}
\begin{proof}[Proof of Proposition \ref{prop:mom}]
Setting for any continuous function $g\colon \R\to\R$
\begin{align}\label{def:h}
\h^g(u)&:=\ \frac{2}{\sigma^2(u)\rho_b(u)}\int_\R g(y)\rho_b(y)\left(\mathds{1}\{u>y\}-F_b(u)\right)\d y, \hspace*{2em} u\in\R,\\\nonumber
G^g(z)&:=\ \int_0^z \h^g(u)\d u, \hspace*{16.5em}\ \ z\geq 0,
\end{align}
we can apply It\^o's formula to $G^{fb_0}(\cdot)$ and $\X$ to obtain
\begin{align*}\nonumber
\int_{X_0}^{X_t}\h^{fb_0}(u)\d u &= \int_0^t \h^{fb_0}(X_s)b(X_s)\d s+\int_0^t \h^{fb_0}(X_s)\sigma(X_s)\d W_s+\frac{1}{2}\int_0^t(G^{fb_0})''(X_s)\d s\\
&=\int_0^t \h^{fb_0}(X_s)\sigma(X_s)\d W_s+\int_0^t\Big((fb_0)(X_s)-\E_b\left[(fb_0)(X_0)\right]\Big)\d s.
\end{align*}
This gives the representation \eqref{eq:martre} for the specifications $\Ma_t^f:=-\int_0^t \h^{fb_0}(X_s)\sigma(X_s)\d W_s$ and $\rd_t^f:=\int_{X_0}^{X_t}\h^{fb_0}(u)\d u$. 
The next step consists in bounding the function $\h^{fb_0}(\cdot)$. Note first that the conditions on the class $\Sigma$ ensure that there exists a constant $K=K(\C,A,\gamma,\overline\nu,\underline\nu)$ such that, for any $b\in\Sigma(\C,A,\gamma,\sigma)$,
\[
\sup_{x\geq 0}\frac{1-F_b(x)}{\sigma^2(x)\rho_b(x)}\ \leq\ K\quad\text{ and } \quad\sup_{x\leq0}\frac{F_b(x)}{\sigma^2(x)\rho_b(x)}\ \leq\ K.
\] 
For $y\in[0,A]$, we have
\begin{eqnarray*}
\sigma^{2}(y)\rho_b(y)&=& C_{b,\sigma}^{-1}\exp\left(2\int_0^y\frac{b(v)}{\sigma^2(v)}\d v\right)\ \geq\ C_{b,\sigma}^{-1}\exp\left(-2\int_0^y\frac{|b(v)|}{\sigma^2(v)}\d v\right)\\
&\geq& C_{b,\sigma}^{-1}\exp\left(-2\underline\nu^{-2}\C\int_0^y(1+v)\d v\right)\ \geq \ C_{b,\sigma}^{-1}\exp\left(-2\underline\nu^{-2}\C\int_0^A (1+v)\d v\right)\\
&=& C_{b,\sigma}^{-1} \e^{-\underline\nu^{-2}\C(2A+A^2)}.
\end{eqnarray*} 
Since the same arguments apply to $y\in [-A,0]$, it holds
\[(\sigma^2\rho_b)^{-1}(y)\ \leq\  C_{b,\sigma} \e^{\underline\nu^{-2}\C(2A+A^2)},\quad y\in [-A,A].\]
We start with analysing the general case. Let $f\in\mathcal F\cup\overline{\mathcal F}$, and note that $\lebesgue(\supp(f))\leq \mathcal S$.
For any $u\in\R$ and the function $\h^{fb_0}$ defined according to \eqref{def:h}, we have
\begin{eqnarray*}
|\h^{fb_0}(u)|^2&\leq&4\int_\R f^2(y)\d y\int_{\supp(f)}
C^2(1+|y|^\eta)^2\rho_b^2(y)\frac{(\mathds{1}\{u>y\}-F_b(u))^2}{\sigma^4(u)\rho_b^2(u)}\d y\\
&=& 4C^2\|f\|_{L^2(\lebesgue)}^2\Bigg\{\frac{(1-F_b(u))^2}{\sigma^4(u)\rho_b^2(u)}\int_{-\infty}^u\mathds{1}\{y\in\supp(f)\}(1+|y|^\eta)^2\rho_b^2(y)\d y\\
&&\hspace*{8em} +\ \frac{F_b^2(u)}{\sigma^4(u)\rho_b^2(u)}\ \int_u^{\infty}\mathds{1}\{y\in\supp(f)\}(1+|y|^\eta)^2\rho_b^2(y)\d y\Bigg\}.
\end{eqnarray*}
Now, for $u>A$,  
\begin{eqnarray*}
|\h^{fb_0}(u)|^2\ &\leq&
 4C^2\|f\|_{L^2(\lebesgue)}^2\Bigg\{2K^2\mathcal L\mathcal S \left(1+\sup_{x\in\R} |x|^{2\eta}\rho_b(x)\right)\\
&&\hspace*{0.5em} +\ \underline\nu^{-2}\int_u^{\infty}\mathds{1}\{y\in\supp(f)\}\exp\left(4\int_u^y\frac{b(z)}{\sigma^2(z)}\d z\right)\ \left(2+2|y|^{2\eta}\right)\d y\Bigg\} \\
&\leq&  4C^2\|f\|_{L^2(\lebesgue)}^2\Bigg\{2K^2\mathcal L\mathcal S \left(1+\sup_{x\in\R} |x|^{2\eta}\rho_b(x)\right)\\
&&\hspace*{0.5em} +\ \underline\nu^{-2}\int_u^{\infty}\mathds{1}\{y\in\supp(f)\}\max\{2^{2\eta},2\}\e^{-4\gamma(y-u)} \left(1+|y-u|^{2\eta}\right)\d y\\
&&\hspace*{0.5em} +\ \underline\nu^{-2}\int_u^{\infty}\mathds{1}\{y\in\supp(f)\}\max\{2^{2\eta},2\} \e^{-4\gamma(y-u)} u^{2\eta}\d y\Bigg\} \\
&\leq&  4C^2\|f\|_{L^2(\lebesgue)}^2\Bigg\{2K^2\mathcal L\mathcal S \left(1+\sup_{x\in\R} |x|^{2\eta}\rho_b(x)\right)\\
&&\hspace*{0.5em} +\ \underline\nu^{-2}\mathcal S \max\{2^{2\eta},2\} \sup_{x\geq 0}\left(\exp(-4\gamma x)x^{2\eta}\right)+ \underline\nu^{-2}\mathcal S\max\{2^{2\eta},2\}(1+u^{2\eta})\Bigg\} \\
&\leq&\ 4C^2\|f\|_{L^2(\lebesgue)}^2\max\{2^{2\eta},2\}\
\mathcal S\Bigg\{2K^2\mathcal L \left(1+\sup_{x\in\R} |x|^{2\eta}\rho_b(x)\right)\\
&&\hspace*{.5em}+\ \underline\nu^{-2}\left(\sup_{x\geq0}\left(\exp(-4\gamma x)x^{2\eta}\right)+1+u^{2\eta}\right)\Bigg\}.
\end{eqnarray*}
The case $u<-A$ can be treated analogously. 
For $-A\leq u\leq A$, it holds
\begin{eqnarray*}
|\h^{fb_0}(u)|^2\ &\leq & 4C^2\|f\|_{L^2(\lebesgue)}^2 \sup_{-A\leq x\leq A}\frac{4\mathcal L\mathcal S}{\sigma^4(x)\rho_b^2(x)}
\ \left(1+\sup_{x\in\R}|x|^{2\eta}\rho_b(x)\right)\\
&\leq& 16C^2\|f\|_{L^2(\lebesgue)}^2\mathcal L\mathcal S C_{b,\sigma}^2 \e^{2\underline\nu^{-2}}\C(2A+A^2)\left(1+\sup_{x\in\R}|x|^{2\eta}\rho_b(x)\right).
\end{eqnarray*}
Thus, for any $u\in\R$, $f\in\mathcal F\cup\overline{\mathcal F}$ and $\overline{\Lambda}_{\prox}$ defined according to \eqref{def:lprox}, 
\begin{equation}\label{Abschaetzung_h}
|\h^{fb_0}(u)|^2\ \leq\ \overline{\Lambda}_{\operatorname{\prox}}^2\mathcal S\|f\|_{L^2(\lebesgue)}^2\ \left(1+|u|^{2\eta}\right).
\end{equation}
For any $p\geq 2$, it now follows from \eqref{4.2}, \eqref{Abschaetzung_h} and \eqref{moment_bounds_invariant_measure} that
\begin{eqnarray}\label{c62}
\E_b\left[\Big|\frac{1}{\sqrt t}\Ma_t^{f}\Big|^p\right] &\leq& 
\bdg^pp^{p/2}\E_b\left[\left(\frac{1}{t}\int_0^t(\h^{fb_0})^2(X_s)\sigma^2(X_s)\d s\right)^{p/2}\right]\\\nonumber
&\leq&\overline\nu^p\bdg^p p^{p/2}\left(\overline{\Lambda}_{\operatorname{\prox}}^2\mathcal S\|f\|_{L^2(\lebesgue)}^2\right)^{p/2}t^{-p/2}\E_b\left[\left(\int_0^t\left(1+|X_s|^{2\eta}\right)\d s\right)^{p/2}\right]\\\nonumber
&\leq& \|f\|_{L^2(\lebesgue)}^p {\overline\nu}^p\bdg^p p^{p/2}\left(\overline{\Lambda}_{\operatorname{\prox}}^2\mathcal S\right)^{p/2} \left(1+\mom^{2\eta }(\eta p)^{2\eta}\right)^{p/2}\\\nonumber
&\leq& \|f\|_{L^2(\lebesgue)}^p {\overline\nu}^p\bdg^p\left(\overline{\Lambda}_{\operatorname{\prox}}^2\mathcal S\right)^{p/2}p^{\eta p+p/2}\left(1+(\mom\eta)^{2\eta }\right)^{p/2}.
\end{eqnarray}
For $1\leq p<2$, one obtains
\begin{eqnarray*}
\E_b\left[\Big|\frac{1}{\sqrt t}\Ma_t^{f}\Big|^p\right] &\leq& 
(2p)^{p/2}\bdg^p\sqrt{\E_b\left[\left(\frac{1}{t}\int_0^t(\h^{fb_0})^2(X_s)\sigma^2(X_s)\d s\right)^{p}\right]}\\
&\leq&(2p)^{p/2}{\overline\nu}^p\bdg^p \left(\overline{\Lambda}_{\operatorname{\prox}}^2\mathcal S\|f\|_{L^2(\lebesgue)}^2\right)^{p/2}t^{-p/2}\sqrt{\E_b\left[\left(\int_0^t(1+|X_s|^{2\eta})\d s\right)^{p}\right]}\\
&\leq&(2p)^{p/2} \|f\|_{L^2(\lebesgue)}^p {\overline\nu}^p\bdg^p\left(\overline{\Lambda}_{\operatorname{\prox}}^2\mathcal S\right)^{p/2} \left(1+\mom^{2\eta }(\eta 2p)^{2\eta}\right)^{p/2}\\
&\leq&(2p)^{p/2+\eta p}|f\|_{L^2(\lebesgue)}^p {\overline\nu}^p\bdg^p\left(\overline{\Lambda}_{\operatorname{\prox}}^2\mathcal S\right)^{p/2}\left(1+(\mom\eta)^{2\eta }\right)^{p/2}.
\end{eqnarray*}
For bounding the remainder term, we start by noting that \eqref{Abschaetzung_h} implies the upper bound 
$\sup_{f\in\FF}|\h^{fb_0}(u)| \leq\overline{\Lambda}_{\operatorname{\prox}}\mathcal S(1+|u|^{\eta})$.
Consequently, for any $p\geq 1$,
\begin{align*}
\left\|\sup_{f\in\FF}\left|\int_{X_0}^{X_t}\h^{fb_0}(u)\d u\right|\right\|_p
&\leq \mathcal S \overline\Lambda_{\operatorname{\prox}}\left( 2\|X_0\|_p + \frac{2}{\eta +1}\|X_0^{\eta+1}\|_p\right)\\
&\leq 2\mathcal S \overline\Lambda_{\operatorname{\prox}}\left(\mom p + \frac{1}{\eta+1}\mom^{\eta+1}(\eta+1)^{\eta+1} p^{\eta+1}\right)\\
&\leq 4\mathcal S\overline\Lambda_{\operatorname{\prox}}\max\left\{\mom^{\eta+1},\ 1\right\}(\eta+1)^{\eta} p^{\eta+1}.
\end{align*}

We now turn to the particular case $b_0=b$.
For this case, one could use the above results with $\eta=1$.
However, one obtains better estimates by exploiting the relation between $\rho_b$ and $b$. 
We start with considering the martingale part. 
Let $f,g\in\mathcal F$, and let $x_f,x_g\in\R$ such that $\supp(f-g)\subset [x_f,x_f+\mathcal S]\cup [x_g,x_g+\mathcal S]$. 
Then, for any $u\in\R$,
\begin{eqnarray*}
&&|\h^{(f-g)b}(u)|^2\\
&&\hspace*{2em}\leq 4\int_\R (f-g)^2(y)\d y\ \int_\R \mathds{1}\{y\in\supp(f-g)\}b^2(y)\rho_b^2(y)\frac{(\mathds{1}\{u>y\}-F_b(u))^2}{\sigma^4(u)\rho_b^2(u)}\d y\\
&&\hspace*{2em}\leq 4\|f-g\|_{L^2(\lebesgue)}^2\Bigg\{ \int_\R \mathds{1}\{y\in [x_f,x_f+\mathcal S]\}b^2(y)\rho_b^2(y)\frac{(\mathds{1}\{u>y\}-F_b(u))^2}{\sigma^4(u)\rho_b^2(u)}\d y\\
&&\hspace*{9em}+ \int_\R \mathds{1}\{y\in [x_g,x_g+\mathcal S]\}b^2(y)\rho_b^2(y)\frac{(\mathds{1}\{u>y\}-F_b(u))^2}{\sigma^4(u)\rho_b^2(u)}\d y\Bigg\}.
\end{eqnarray*}
For $u>A$, and any $x\in\R,$ exploiting the relation $(\sigma^2\rho_b)'=2b\rho_b$, it holds 
\begin{align*}
&\int_\R \mathds{1}\{y\in[x,x+\mathcal S]\}b^2(y)\rho_b^2(y)\frac{(\mathds{1}\{u>y\}-F_b(u))^2}{\sigma^4(u)\rho_b^2(u)}\d y\\
&\qquad\leq\ \frac{(1-F_b(u))^2}{\sigma^4(u)\rho_b^2(u)}\int_{-\infty}^u \mathds{1}\{y\in[x,x+\mathcal S]\}|b^2(y)\rho^2_b(y)|\d y\\
&\hspace*{4em}+\frac{1}{4\sigma^4(u)\rho_b^2(u)}\int_u^\infty \mathds{1}\{x\leq y\leq x+\mathcal S\}|b(y)||\left(\sigma^4\rho^2_b\right)'(y)|\sigma^{-2}(y)\d y\\ 
&\qquad\leq \ \sup_{z\geq 0}\frac{(1-F_b(z))^2}{\sigma^4(z)\rho_b^2(z)}\frac{1}{4}\ML^2\mathcal S
+\mathds{1}\{u\leq x\}\frac{\underline\nu^{-2}\mathbb C(\sigma^2\rho_b)^2(x)}{4(\sigma^2\rho_b)^2(u)}(1+x+\mathcal S) \\
&\hspace*{4em}+\, \mathds{1}\{u>x\}\frac{\underline\nu^{-2}\mathbb C(\sigma^2\rho_b)^2(u)}{4(\sigma^2\rho_b)^2(u)}(1+u+\mathcal S)\\ 
&\qquad\leq\ \frac{K^2\ML^2\mathcal S}{4}
+\mathds{1}\{u\leq x\}\frac{\underline\nu^{-2}\mathbb C}{4}\left(\e^{-4\gamma(x-u)}(1+x-u+\mathcal S)+\e^{-4\gamma (x-u)}u\right)\\
&\hspace*{4em}+\, \frac{\underline\nu^{-2}\mathbb C}{4}(1+u+\mathcal S)\\ 
&\qquad\leq\ \frac{K^2\ML^2\mathcal S}{4}+\frac{\underline\nu^{-2}\mathbb C}{2}\left(1+ u +\mathcal S + \sup_{z\geq 0}\left(\exp(-4\gamma z)z\right)\right). 
\end{align*}
For $u<-A$,
\begin{align*}
&\int_\R \mathds{1}\{y\in(x,x+\mathcal S)\}b^2(y)\rho_b^2(y)\frac{(\mathds{1}\{u>y\}-F_b(u))^2}{\sigma^4(u)\rho_b^2(u)}\d y\\
&\qquad\leq\ \frac{1}{4\sigma^4(u)\rho_b^2(u)}\int_{-\infty}^u \mathds{1}\{y\in(x,x+\mathcal S)\}|b(y)||\left(\sigma^4\rho^2_b\right)'(y)|\sigma^{-2}(y)\d y\\
&\hspace*{4em}+ \, \frac{F_b^2(u)}{\sigma^4(u)\rho_b^2(u)}\int_u^\infty \mathds{1}\{x\leq y\leq x+\mathcal S\}b^2(y)\rho_b^2(y)\d y\\ 
&\qquad\leq\ \mathds{1}\{u\geq x+\mathcal S\}\frac{\underline\nu^{-2}\mathbb C(\sigma^2\rho_b)^2(x+\mathcal S)}{4\rho_b^2(u)}(1+|x|) \\
&\hspace*{4em}+\, \mathds{1}\{u<x+\mathcal S\}\frac{\underline\nu^{-2}\mathbb C(\sigma^2\rho_b)^2(u)}{4(\sigma^2\rho_b)^2(u)}(1+|u|+\mathcal S) +  \sup_{z\leq 0}\frac{F_b^2(z)}{\sigma^4(z)\rho_b^2(z)}\frac{1}{4}\ML^2\mathcal S\\
&\qquad\leq\ \mathds{1}\{u\geq x+\mathcal S\}\Bigg(\frac{\underline\nu^{-2}\mathbb C}{4}\e^{-4\gamma(u-(x+\mathcal S)}(1+u-(x+\mathcal S)) + \frac{\underline\nu^{-2}\mathbb C}{4}\e^{-4\gamma(x-u)}(|u|+\mathcal S)\Bigg)\\
&\hspace*{4em}+\, \frac{\underline\nu^{-2}\mathbb C}{4}(1+|u|+\mathcal S)+\frac{K^2\ML^2\mathcal S}{4}\\ 
&\qquad\leq\ \frac{K^2\ML^2\mathcal S}{4}+\frac{\underline\nu^{-2}\mathbb C}{2}\left(1+ |u| +\mathcal S + \sup_{z\geq 0}\left(\exp(-4\gamma z)z\right) \right). 
\end{align*}
Finally, for $u\in[-A,A]$,
\begin{align*}
|\h^{(f-g)b}(u)|^2\ &\leq\ 4\|f-g\|_{L^2(\lebesgue)}^2\bigg\{\frac{(1-F_b(u))^2}{\sigma^4(u)\rho_b^2(u)}\int_{-\infty}^u\mathds{1}\{y\in\supp(f-g)\}b^2(y)\rho_b^2(y)\d y\\
&\hspace*{6em} +\ \frac{F_b^2(u)}{\sigma^4(u)\rho_b^2(u)}\ \int_u^{\infty}\mathds{1}\{y\in\supp(f-g)\}b^2(y)\rho_b^2(y)\d y\bigg\}\\
&\leq\ 4\|f-g\|_{L^2(\lebesgue)}^2 \sup_{-A\leq z\leq A}\frac{\ML^2\mathcal S}{2\sigma^4(z)\rho_b^2(z)}\\
& \leq\ 4\|f-g\|_{L^2(\lebesgue)}^2\ML^2\mathcal S C^2_{b,\sigma} \e^{2\underline\nu^{-2}\C(2A+A^2)}.
\end{align*}
Summing up,
$|\h^{(f-g)b}(u)|^2 \leq\overline{\Gamma}^2_{\prox} \|f-g\|_{L^2(\lebesgue)}^2(1+|u|+\mathcal S)$.
The same arguments give, for any $f\in\mathcal F$,
\[|\h^{fb}(u)|^2\ \leq\  \overline{\Gamma}_{\prox}^2 \|f\|_{L^2(\lebesgue)}^2(1+|u|+\mathcal S).\]
Similarly to \eqref{c62}, we can conclude for all $f\in\mathcal F\cup\overline{\mathcal F}$, $p\geq 2$,
\begin{eqnarray*}
 \E_b\left[\Big|\frac{1}{\sqrt t}\Ma_t^{f}\Big|^p\right] 
&\leq& \bdg^pp^{p/2}t^{-p/2}\E_b\left[\left(\int_0^t(\h^{fb})^2(X_s)\sigma^2(X_s)\d s\right)^{p/2}\right]\\
&\leq& {\overline\nu}^p\bdg^pp^{p/2}t^{-p/2}\E_b\left[\left(\int_0^t\overline{\Gamma}^2_{\prox} \|f\|_{L^2(\lebesgue)}^2(1+|X_s|+\mathcal S)\d s\right)^{p/2}\right]\\
&\leq&\overline{\Gamma}_{\prox}^p \|f\|_{L^2(\lebesgue)}^p {\overline\nu}^p\bdg^pp^{p/2}t^{-p/2} \E_b\left[\left(\int_0^t(1+|X_s|+\mathcal S)\d s\right)^{p/2}\right]\\
&\leq&\overline{\Gamma}_{\prox}^p \|f\|_{L^2(\lebesgue)}^p {\overline\nu}^p\bdg^pp^{p}(1+\mathcal S+\mom)^{p/2}.
\end{eqnarray*}
For $1\leq p<2$, it holds
\begin{eqnarray*}
\E_b\left[\Big|\frac{1}{\sqrt t}\Ma_t^{f}\Big|^p\right]
&\leq& (2p)^{p/2}\bdg^p t^{-p/2}\sqrt{\E_b\left[\left(\int_0^t(\h^{fb})^2(X_s)\sigma^2(X_s)\d s\right)^{p}\right]}\\
&\leq& (2p)^{p/2}  {\overline\nu}^p\bdg^p t^{-p/2}\sqrt{\E_b\left[\left(\int_0^t\overline{\Gamma}^2_{\prox} \|f\|_{L^2(\lebesgue)}^2(1+|X_s|+\mathcal S)\d s\right)^{p}\right]}\\
&\leq&2^{p/2}\overline{\Gamma}_{\prox}^p \|f\|_{L^2(\lebesgue)}^p {\overline\nu}^p\bdg^pp^{p}(1+\mathcal S+\mom)^{p/2}.
\end{eqnarray*}
Hence, we have shown for any $p\geq 1$
\[
\left(\E_b\left[\Big|\frac{1}{\sqrt t}\Ma_t^{f}\Big|^p\right]\right)^{\p}\ \leq\ \sqrt 2\overline{\Gamma}_{\prox} \|f\|_{L^2(\lebesgue)} \overline\nu\bdg (1+\mom+\mathcal S)^{1/2} p.
\]
For bounding the remainder term, let $f\in\mathcal F$. We start by decomposing $|\h^{fb}|\leq 2A_1+2A_2$, with
\[
A_1(u)\ :=\ \frac{1-F_b(u)}{\sigma^2(u)\rho_b(u)}\ \left|\int_{-\infty}^u (fb)(y)\rho_b(y)\d y\right|,\,\,\,
A_2(u)\ :=\ \frac{F_b(u)}{\sigma^2(u)\rho_b(u)}\ \left|\int_{u}^\infty (fb)(y)\rho_b(y)\d y\right|.
\]
For $u\geq 0$,
\begin{eqnarray*}
A_1(u)&\leq& K\Big(\int_{-\infty}^{-A}|f(y)|\frac{(\sigma^2\rho_b)'(y)}{2}\d y + \int_{-A}^A|f(y)b(y)|\rho_b(y)\d y\\
&&\hspace*{2em}- \mathds{1}\{u>A\}\int_A^u \frac{(\sigma^2\rho_b)'(y)}{2}|f(y)|\d y\Big)\\
&\leq& K U\left(\frac{\overline\nu^2\ML}{2} + \mathbb C(1+A)\mu_b([-A,A]) + \overline\nu^2\ML   \right)\ \leq\ K U\left(2\overline\nu^2\ML + \mathbb C(1+A)\right),\\
A_2(u)&\leq& U\left(\mathds{1}\{u\leq A\} C_{b,\sigma} \e^{\underline\nu^{-2}\C(2A+A^2)}\left(\int_u^A \mathbb C(1+A)\ML \d y -\int_A^\infty\frac{(\sigma^2\rho_b)'(y)}{2}\d y \right)\right)\\
&&\hspace*{0em}- U\mathds{1}\{u>A\}\frac{1}{\sigma^2(u)\rho_b(u)}\int_u^\infty \frac{(\sigma^2\rho_b)'(y)}{2}\d y \\
&\leq&  U\left(C_{b,\sigma} \e^{\underline\nu^{-2}\C(2A+A^2)}\left(A \mathbb C(1+A)\ML +\frac{\overline\nu^2\ML}{2}\right)\right) + \frac{1}{2}U.
\end{eqnarray*}
For $u\leq 0$,
\begin{eqnarray*}
A_1(u)&\leq&  U\left(\mathds{1}\{u\geq - A\} C_{b,\sigma} \e^{\underline\nu^{-2}\C(2A+A^2)}\left(\int_{-A}^u \mathbb C(1+A)\ML \d y +\int_{-\infty}^{-A}\frac{(\sigma^2\rho_b)'(y)}{2}\d y \right)\right)\\
&&\hspace*{0em} +\,\, U\mathds{1}\{u<-A\}\frac{1}{\sigma^2(u)\rho_b(u)}\int_{-\infty}^u \frac{(\sigma^2\rho_b)'(y)}{2}\d y \\
&\leq& U\left(C_{b,\sigma} \e^{\underline\nu^{-2}\C(2A+A^2)}\left(A \mathbb C(1+A)\ML +\frac{\overline\nu^2\ML}{2}\right)\right) + \frac{1}{2} U,\\
A_2(u)&\leq& K U\left(\mathds{1}\{u\leq -A\}\int_{u}^{-A}\frac{(\sigma^2\rho_b)'(y)}{2}\d y +  \int_{-A}^A|b(y)|\rho_b(y)\d y-\int_A^\infty \frac{(\sigma^2\rho_b)'(y)}{2}\d y\right)\\
&\leq& K U\left(\overline\nu^2\ML  + \mathbb C(1+|A|)\mu_b([-A,A]) + \frac{\overline\nu^2\ML}{2}\right)\ \leq\ K U\left(2\overline\nu^2\ML + \mathbb C(1+A)\right).
\end{eqnarray*}
We have thus shown that 
\begin{eqnarray*}\sup_{f\in\mathcal F}
\left\|\h^{fb}\right\|_\infty\ &\leq&\ 4K U\left(2\overline\nu^2\ML + \mathbb C(1+A)\right)\\
&& +\,\, 4 U\left(C_{b,\sigma} \e^{\underline\nu^{-2}\C(2A+A^2)}\left(A \mathbb C(1+A)\ML +\frac{\overline\nu^2\ML}{2}\right)\right) + 2 U,\end{eqnarray*}
and, finally, 
\begin{eqnarray*}
\left(\E_b\left[\left\|\rd_t^{f}\right\|_\FF^p\right]\right)\p
&=&\left(\E_b\left[\bigg\|\int_{X_0}^{X_t}\h^{fb}(u)\d u\bigg\|_\FF^p\right]\right)\p\ \leq\ \sup_{f\in\mathcal F}\left\|\h^{fb}\right\|_\infty\left(\E_b\left[\left|X_t-X_0\right|^p\right]\right)\p\\
&\leq&2 \mom\,p\sup_{f\in\mathcal F}\left\|\h^{fb}\right\|_\infty\ \leq\ \Gamma_{\prox} p.
\end{eqnarray*}
\end{proof}

\medskip

\begin{proof}[Proof of Theorem \ref{theo:bernsup}]
Under the given assumptions, Proposition \ref{prop:mom} implies the decomposition 
\[\G^{b_0}_t(f)\ =\ t^{-1/2}\Ma_t^f+t^{-1/2}\rd_t^f,\quad t>0.\]
For $b_0\equiv 1$, we further obtain, for any $p\geq1$, $f\in\FF\cup\overline\FF$,
\[
\left(\E_b\left[|\Ma_t^{f}|^p\right]\right)\p\ \leq\ \ps_1\sqrt{pt\mathcal S}\|f\|_{L^2(\lebesgue)}, \quad 
\left(\E_b\left[\|\rd_t^{fb}\|_\FF^p\right]\right)\p\ \leq\ \ps_2 p\mathcal{S}.\]
This corresponds to the case $\alpha=2$ in Theorem \ref{theo:1} which then yields $(\mathbf{I})$.
For $b_0=b$, the upper bounds on the $p$-th moments, $p\geq 1$, of the martingale and remainder term are specified as
\[
\left(\E_b\left[|\Ma_t^{f}|^p\right]\right)\p\ \leq\ \ps_1^b p^{\frac{3}{2}}\sqrt {t\mathcal S}\|f\|_{L^2(\lebesgue)},\  f\in\mathcal F\cup \overline{\mathcal F},\qquad
\left(\E_b\left[\|\rd_t^{fb}\|_\FF^p\right]\right)\p\ \leq\ \ps^b_2 p, \ f\in\FF.\]
Here, we combined the upper bound for the moments of the  martingale part for the general case (letting $\eta=1$) in Proposition \ref{prop:mom} 
with the upper bound on the moments of the remainder term for the specific drift part (equation $(\mathbf{II})$ of the Proposition). 
The upper bounds correspond to the case $\alpha=\frac{2}{3}$ in Theorem \ref{theo:1}. The assertion follows together with Lemma \ref{EntropyIntegralLebesgueMetric}.
\end{proof}

\medskip

\begin{proof}[Proof of Proposition \ref{theo:mart}]
The proof substantially relies on Proposition \ref{prop:mom} and Theorem \ref{theo:1}.

\medskip
\noindent
\textbf{\textsl{Martingale approximation.}}
Our first step is the martingale approximation of the non-martingale part of 
\[\H_t(f)\ =\ \sqrt t\left(\frac{1}{t}\int_0^t f(X_s)b(X_s)\d s - \int(fb)d\mu_b + \frac{1}{t}\int_0^t f(X_s)\sigma(X_s)dW_s\right).\]
Proposition \ref{prop:mom} gives the representation $\sqrt t\H_t(f)=\Ma_t^f+\rd_t^f+\int_0^tf(X_s)\d W_s$,
and equation $(\mathbf{II})$ yields, for any $p\geq 1$,
\begin{eqnarray*}
\left(\E_b\left[|\Ma_t^{f}|^p\right]\right)\p&\leq& \Phi_1\sqrt t p\|f\|_{L^2(\lebesgue)},\quad \text{ for } \Phi_1:=\sqrt 2\overline{\Gamma}_{\operatorname{\prox}}\overline\nu \bdg\sqrt{1+\mom+\mathcal S},\\
\left(\E_b\left[\left(\sup_{f\in\FF}|\rd^{fb}_t|\right)^p\right]\right)\p &\leq& \Phi_2p,\hspace*{6em} \text{ for }\Phi_2:=\Gamma_{\operatorname{\prox}}.
\end{eqnarray*}

\medskip
\noindent
\textbf{\textsl{Application of Theorem \ref{theo:1}.}}
Plugging the above estimates of the moments of $\Ma_t^f$ and $\rd_t^f$ into the moment bound of Theorem \ref{theo:1}, one gets together with Lemma \ref{EntropyIntegralLebesgueMetric}, for any $p\geq 1$,
\begin{eqnarray*}
\frac{1}{\sqrt t}\left(\E_b\left[\big\|\Ma_t^{f}+\rd_t^{f}\big\|_\FF^p\right]\right)\p\ &\leq& C_1\sum_{k=0}^{\infty}  E(F_k,\e \Phi_1\ L^2,1) \exp\left(-\frac{k}{2}\right)  + 12\Phi_1 p\V+2\frac{\Phi_2p}{\sqrt t}\\
&&\hspace*{3em}+ \sqrt t C\U(1+2\mom)\exp\left(-\frac{\Lambda t}{2\e\mom}\right)\ \leq\ \tilde{\Phi}_t^b(p),
\end{eqnarray*}
with 
\begin{align*}
\tilde{\Phi}_t^b(p)&:=\ 6C_1v\V\e \Phi_1\left(1+ \log\left(\frac{\AA}{\V}\sqrt{\mathcal S + p\Lambda t}\right)\right) + 12\Phi_1p\V+2\frac{\Phi_2p}{\sqrt t}\\
&\hspace*{10em}+ \sqrt t C\U(1+2\mom)\exp\left(-\frac{\Lambda t}{2\e\mom}\right).
\end{align*}

\medskip
\noindent
\textbf{\textsl{Bounding the $p$-th moments of the original stochastic integral term.}}
Corollary \ref{cor:local_time} yields a constant $\tilde\kappa$ such that, for any $p\geq 1$ and $t\geq 1$,
\begin{eqnarray*}
&&\max\left\{\left(\E_b\left[\sup_{a\in\mathbb Q}\left|\frac{1}{t}L_t^a(X)\right|^{p/2}\right]\right)^{\frac{1}{p}}, \left(\E_b\left[\sup_{a\in\mathbb Q}\left|\frac{1}{t} L_t^a(X)\right|^{p}\right]\right)^{\frac{1}{2p}}\right\}\\
&&\hspace*{20em}\leq \tilde\kappa\left(\sqrt p + t^{-1/4}\sqrt{\log t} + p^{1/4}t^{-1/4}\right).
\end{eqnarray*}
Consequently, as in the proof of Theorem \ref{improved_version}, for any $p\geq 1$,
\begin{eqnarray*}
\left(\E_b\left[\bigg|\frac{1}{\sqrt t}\int_0^tf(X_s)\sigma(X_s)\d W_s\bigg|^p\right]\right)\p\ \leq\
\overline\nu\bdg\sqrt p\|f\|_{L^2(\lebesgue)}  \tilde\kappa\left(\sqrt p + t^{-1/4}\sqrt{\log t} + p^{1/4}t^{-1/4}\right).
\end{eqnarray*}
In particular, this last estimate implies that the $p$-th moments are uniformly bounded over $\FF$ and that the process $\left(t^{-1/2}\int_0^tf(X_s)\sigma(X_s)\d W_s\right)_{f\in\FF}$ exhibits a subexponential tail behaviour. 
Precisely, we have 
\begin{equation*}
\P_b\left(\Big|t^{-1/2}\int_0^t(f-g)(X_s)\sigma(X_s)\d W_s\Big|\ \geq\ d(f,g) u \right)\ \leq\ \exp(-u),
\end{equation*}
for $d(f,g):=\max\left\{\e\Lambda_3,1\right\}\|f-g\|_{L^2(\lebesgue)}$, $\Lambda_3:=\max\{4\tilde\kappa\overline\nu\bdg + 2\tilde\kappa\overline\nu\bdg\sup_{t\geq 1 }t^{-1/4}\sqrt{\log t},1\}$.
Following the scheme of the proof of Theorem \ref{theo:1}, we apply the chaining procedure from Proposition \ref{thm:dirk} locally and obtain, for all $p,q\geq 1$, $k\in\N_0$,
\begin{eqnarray*}
\frac{1}{\sqrt t}\left\|\sup_{f\in\FF_k}\int_0^t f(X_s)\sigma(X_s)\d W_s\right\|_{q}&\leq&
 C_1\int_0^\infty \log N(u,\FF_k,d)\d u\\
&& \hspace*{0em} +\,  2\sup_{f\in\mathcal F_k}\frac{1}{\sqrt t}\bigg\|\int_0^t f(X_s)\sigma(X_s)\d W_s\bigg\|_q\\
&\leq& C_1\int_0^\infty \log N(u,\FF_k,d)\d u + \Lambda_3\V q.
\end{eqnarray*}
From the local result, we deduce, for any $p\geq 1$,
\begin{align*}
&\frac{1}{\sqrt t}\left(\E_b\left[\bigg\|\int_0^tf(X_s)\sigma(X_s)\d W_s\bigg\|_{\FF}^p\right]\right)^{\frac{1}{p}}\\
&\hspace*{2em}\leq\ \frac{1}{\sqrt t}\left(\E_b\left[\bigg\|\int_0^tf(X_s)\sigma(X_s)\d W_s\bigg\|_{\FF_0}^p\right]\right)\p\\
&\hspace*{4em} + \,\sum_{k=1}^\infty \frac{1}{\sqrt t}\left(\E_b\left[\bigg\|\int_0^tf(X_s)\sigma(X_s)\d W_s\bigg\|_{\FF_k}^{2p}\right]\right)^{\frac{1}{2p}}\left(\P_b(A_k)\right)^{\frac{1}{2p}}\\
&\hspace*{2em}\leq\ C_1\sum_{k=0}^\infty E(\mathcal F_k, d,1)\exp\left(-\frac{k}{2}\right) + 6\Lambda_3\V p.
\end{align*}
The entropy integrals can be bounded independently of $k\in\N_0$ by means of Lemma \ref{EntropyIntegralLebesgueMetric}, allowing us to conclude finally that, for any $p\geq 1$,
\begin{eqnarray*}
\left(\E_b\left[\|\H_t\|_\FF^p\right]\right)\p
&\leq& \tilde \Phi_t^b(p) + \tilde \Psi_t^b(p),
\end{eqnarray*}
with
\[\tilde \Psi_t^b(p) \ :=\ 6C_1 v\V\max\left\{\e\Lambda_3,1\right\}\left(1+ \log\left(\frac{\AA}{\V}\sqrt{\mathcal S + p\Lambda t}\right)\right) + 6\Lambda_3\V p.
\]
\end{proof}

\medskip

\section{Proofs for Section \ref{sec:tp}}
\begin{proof}[Proof of Theorem \ref{theo:cath}]
Note first that, for each $x\in\R$, $\rho_{t,K}(h)(x)-\rho_t^\circ(x)$ is a random variable which is right-continuous in $x$.
Thus, $\|\rho_{t,K}(h)-\rho_t^\circ\|_\infty\ =\ \sup_{x\in\mathbb{Q}}|\rho_{t,K}(h)(x)-\rho_t^\circ(x)|$ is also measurable as a supremum over a countable set.
Introduce
\[\Psi_1(x,y)\ :=\ (x-y)\cdot\mathds{1}_{(-\infty,x]}(y),\quad \Psi_2(x,X^t)\ :=\ \int_0^t\mathds{1}_{(-\infty,x]}(X_u)\d X_u,\quad x,y\in\R,\ t>0,\]
and abbreviate $K_h(\cdot):=h^{-1}K(\cdot/h)$.
Using the occupation times formula and Tanaka's formula for diffusion local time, we obtain
\begin{eqnarray*}
\rho_t^\circ(x)&=& t^{-1}\L_t^x(\X) \ = \ 2t^{-1}\left((X_t-x)^--(X_0-x)^-+\int_0^t\mathds{1}_{(-\infty,x]}(X_s)\d X_s\right)\\
&=& 2t^{-1}\left(\Psi_1(x,X_t)-\Psi_1(x,X_0)+\Psi_2(x,X^t)\right),
\end{eqnarray*}
and, since $\int K_h(x-y)\d y =1$,
\begin{eqnarray*}
\rho_{t,K}(h)(x)-\rho_t^\circ(x)&=& t^{-1}\int_\R K_h(x-y)\L_t^y(\X)\d y-  t^{-1}\L_t^x(\X)\\
&=& 2t^{-1}\int_\R K_h(x-y)\left(\Psi_1(y,X_t) -\Psi_1(x,X_t)\right) \d y\\
&&\hspace*{5.5em} + 2t^{-1}\int_\R K_h(x-y)\left(\Psi_1(x,X_0) -\Psi_1(y,X_0)\right) \d y \\
&&\hspace*{5.5em} + 2t^{-1}\int_\R K_h(x-y)\left(\Psi_2(y,X^t)-\Psi_2(x,X^t)\right)\d y\\
&=:& A_{1,x}(t,h)+A_{2,x}(t,h)+B_x(t,h).
\end{eqnarray*}
We start by rewriting
\begin{eqnarray*}
tA_{1,x}(t,h)&=& 2\int_\R K_h(x-y)\left\{(y-X_t)\mathds{1}_{(-\infty,y]}(X_t)-(x-X_t)\mathds{1}_{(-\infty,x]}(X_t)\right\} \d y\\
&=& 2\int_\R K(z)\left\{(x-X_t)\mathds{1}_{(-\infty,x-zh]}(X_t)-(x-X_t)\mathds{1}_{(-\infty,x]}(X_t)\right\}\d z\\
&&\hspace*{15em}-2\int_\R K(z)zh\mathds{1}_{(-\infty,x-zh]}(X_t)\d z.
\end{eqnarray*}
Note that $|(x-X_t)\left(\mathds{1}_{(-\infty,x-zh]}(X_t)-\mathds{1}_{(-\infty,x]}(X_t)\right)|\leq |z|h$ for $h\geq0$.
Thus, for any $x\in\mathbb{Q}$, $A_{1,x}(t,h)\leq 4ht^{-1}\int_\R|K(z)z|\d z$.
Since $A_{2,x}(t,h)$ can be treated analogously, it follows 
\begin{equation}\label{eq:A}
\sup_{x\in\mathbb{Q}}|A_{1,x}(t,h)+A_{2,x}(t,h)|\ \leq \ 8ht^{-1}\int_\R |K(z)z|\d z.
\end{equation} 
It remains to consider $B_x(t,h)$. For any fixed $x\in\mathbb{Q}$, we have 
\begin{align}\nonumber
tB_x(t,h)
&=\ 2\int_\R K_h(x-y)\int_0^t\left\{\mathds{1}_{(-\infty,y]}(X_s)-\mathds{1}_{(-\infty,x]}(X_s)\right\}\d X_s\d y\\\label{eq:fub}
&=\ 2\int_\R K(z)\int_0^t\left\{\mathds{1}_{(-\infty,x-zh]}(X_s)-\mathds{1}_{(-\infty,x]}(X_s)\right\}\d X_s\d z\\\nonumber
&=\ 2\int_0^t\left\{\int_\R K(z)\mathds{1}_{(-\infty,x-zh]}(X_s)\d z-\mathds{1}_{(-\infty,x]}(X_s)\right\}\d X_s\\\nonumber
&=\ 2\int_0^t\left\{\int_\R K_h(y-X_s)\mathds{1}_{(-\infty,x-y+X_s]}(X_s)\d y-\mathds{1}_{(-\infty,x]}(X_s)\right\}\d X_s\\\label{eq:bx}
&=\ 2\int_0^t \left\{K_h\ast \mathds{1}_{(-\infty,x]}-\mathds{1}_{(-\infty,x]}\right\}(X_s)\d X_s.
\end{align}
Here we used a Fubini-type theorem for stochastic integrals (cf.~\cite{kaietal78}), allowing us to change the order of integration in \eqref{eq:fub}.
We proceed by applying Proposition \ref{theo:mart} to the function class
\begin{equation}\label{def:fkh}
\FF_{K,h}\ :=\ \left\{K_h\ast \mathds{1}_{(-\infty,x]}-\mathds{1}_{(-\infty,x]}\colon\ x\in\mathbb{Q}\right\}.
\end{equation}
Following the lines of the proof of Theorem 1 in \cite{gini09}, note first that, for any $x\in\R$, $h>0$,
\[
K_h\ast\mathds{1}_{(-\infty,x]}(\cdot)-\mathds{1}_{(-\infty,x]}(\cdot)
\ =\ H\left(\frac{x-\cdot}{h}\right),\]
for $H(u):=\int_{-\infty}^uK(z)\d z-\mathds{1}_{[0,\infty)}(u)$, $u\in\R$. Since $H$ is of bounded variation, Lemma \ref{lem:entgi} ensures that the entropy condition from Assumption \ref{ann:3} holds true.
Moreover,
\[\sup_{x\in\mathbb{Q}}\left\|K_h\ast\mathds{1}_{(-\infty,x]}(\cdot)-\mathds{1}_{(-\infty,x]}(\cdot)\right\|_\infty\ \leq\ \int_\R |K(z)|\d z+1\ \leq\ 2\|K\|_{L^1(\lebesgue)}\ =:\ \mathbb{K},\]
i.e., $\FF_{K,h}$ is uniformly bounded.
Let us now investigate the $L^2(\lebesgue)$-norm of $\FF_{K,h}$. To this end, fix $x\in\mathbb Q$ and note that, for any $z>0$,
\begin{eqnarray*}
\int \left(\mathds{1}_{(-\infty,x]}(y+z)-\mathds{1}_{(-\infty,x]}(y)\right)^2\d y
&=&\int \left|\mathds{1}_{(-\infty,x]}(y)-\mathds{1}_{(-\infty,x]}(y+z)\right|\d y\\
&=& \int \mathds{1}_{(x-z,x]}(y) \d y\ =\ z.
\end{eqnarray*}
A similar argument for $z\leq 0$ yields
\begin{eqnarray*}
\int \left(\mathds{1}_{(-\infty,x]}(y+z)-\mathds{1}_{(-\infty,x]}(y)\right)^2\d y\ =\ |z|
\end{eqnarray*} for all $z\in\R$.
This bound implies that 
\[\sup_{x\in\mathbb Q} \|K_h\ast\mathds{1}_{(-\infty,x]} - \mathds{1}_{(-\infty,x]} \|_{L^2(\lebesgue)}\ \leq\ \sqrt h \int |K(z)\sqrt z|\d z\] 
since, for any $x\in\mathbb Q$, using Minkowski's integral inequality,
\begin{eqnarray*}
&&\left(\int \left(K_h\ast \mathds{1}_{(-\infty,x]}(y) - \mathds{1}_{(-\infty,x]}(y)\right)^2\d y \right)^{\frac{1}{2}}\\
&&\hspace*{3em}= \left(\int \left(\int K_h(z) \mathds{1}_{(-\infty,x]}(y+z) - \mathds{1}_{(-\infty,x]}(y)\d z\right)^2\d y\right)^{\frac{1}{2}}\\
&&\hspace*{3em}\leq \int|K_h(z)| \left(\int\left( \mathds{1}_{(-\infty,x]}(y+z) - \mathds{1}_{(-\infty,x]}(y)\right)^2\d y\right)^{\frac{1}{2}}\d z\\
&&\hspace*{3em}=  \int|K_h(z)|\sqrt{|z|}\d z\ =\ \sqrt h \int |K(z)|\sqrt{|z|}\d z.
\end{eqnarray*}
Clearly, $\supp(K_h\ast\mathds{1}_{(-\infty,x]} - \mathds{1}_{(-\infty,x]})\subset [x-h/2,x+h/2]$ such that 
\begin{eqnarray*}
\sup_{x\in\mathbb Q} \|K_h\ast\mathds{1}_{(-\infty,x]} - \mathds{1}_{(-\infty,x]}\|_{L^2(\lebesgue)}&\leq&\sqrt{\mathcal S},\\ 
\sup_{x\in\mathbb Q}\lambda(\supp(K_h\ast\mathds{1}_{(-\infty,x]} - \mathds{1}_{(-\infty,x]}))&\leq& \mathcal S:=h\ \max\left\{1, \ \left(\int |K(z)|\sqrt{|z|}\d z\right)^2\right\}.
\end{eqnarray*}
We have thus shown that $\mathcal F_{K,h}$ satisfies Assumption \ref{ann:1}.

However, since the functions $\mathds{1}_{(-\infty,x]}$ are not continuous, Proposition \ref{theo:mart} cannot be applied.
Inspection of the proof shows that continuity is required in order to use Proposition \ref{prop:mom}. 
More precisely, continuity allows to apply It\^o's formula which in turn yields the central representation $\G_t=t^{-1/2}(\Ma_t + \rd_t).$ 
Consequently, Proposition \ref{theo:mart} is applicable once we can show that the same representation is valid for the functions $\mathds{1}_{(-\infty,x]}b$.
For deriving this representation, we need to approximate
\[\int_0^t\mathds{1}\{X_s\leq x\}b(X_s)\d s - \E_b\left[\int_0^t\mathds{1}\{X_s\leq x\}\d X_s\right]
\ =\ \int_0^t\mathds{1}\{X_s\leq x\}b(X_s)\d s - \frac{t}{2}\rho_b(x).\]
Denote $f_x(\cdot):=\mathds{1}\{\cdot\leq x\}\ b(\cdot)$, $x\in\mathbb Q$. We proceed similarly to the proof of Proposition \ref{prop:mom} by setting
\begin{eqnarray*}
\h^{f_x}(u)&:=&\frac{2}{\rho_b(u)}\int f_x(y)\rho_b(y)(\mathds{1}\{u>y\}-F_b(u))\d y\\
&\ =& \mathds{1}\{u> x\} \frac{1}{\rho_b(u)}\rho_b(x)(1-F_b(u)) + \mathds{1}\{u\leq x\}\left(1-\frac{F_b(u)\rho_b(x)}{\rho_b(u)}\right)\\
&\ =& \frac{1}{\rho_b(u)}\rho_b(x)(\mathds{1}\{u> x\} -F_b(u)) + \mathds{1}\{u\leq x\},\\
\h_n(u)&:=&\frac{\rho_b(x)}{\rho_b(u)}(\phi_n(u) -F_b(u)) + (1-\phi_n(u)),
\end{eqnarray*}
for $\phi_n(u)$ denoting a smooth approximation of $\mathds{1}\{u> x\}$, given as
\[\phi_n(u)\ :=\ \frac{n}{\sqrt{2\pi}} \int_{-\infty}^u \exp\left(\frac{-(v-x)^2n^2}{2}\right)\d v\] (cf.~the proof of Proposition 1.11 in \cite{kut04}).
Note that $\lim_{n\to\infty}\phi_n(u)=\mathds{1}\{u>x\}$ and, for any continuous function $g\colon\R\to\R$, it holds 
\begin{equation}\label{mollifier}
\lim_{n\to\infty} \int \phi'_n(u)g(u)\d u\ =\ g(x)
\end{equation}
Set
\[H_n(y)\ :=\ \int_0^y\h_n(u)\d u\quad\text{ and }\quad H(y)\ :=\ \int_0^y \h^{f_x}(u)\d u.\]
Then $H'_n(y)\ =\ \h_n(y)$, and
\begin{eqnarray*}
H''_n(y)&=&-\frac{\rho_b(x)\rho'_b(y)}{\rho_b^2(y)}(\phi_n(y) -F(y)) +\frac{\rho_b(x)}{\rho_b(y)}(\phi'_n(y)-\rho_b(y)) -\phi'_n(y)\\
&=&-\frac{2\rho_b(x) b(y)}{\rho_b(y)}(\phi_n(y) -F(y)) +\frac{\rho_b(x)}{\rho_b(y)}(\phi'_n(y)-\rho_b(y)) -\phi'_n(y)\\
&=&-2\h_n(y)b(y) + 2b(y)(1-\phi_n(y)) + \frac{\rho_b(x)}{\rho_b(y)}(\phi'_n(y)-\rho_b(y)) -\phi'_n(y).
\end{eqnarray*}
It\^o's formula yields
\begin{align*}
&H_n(X_t)-H_n(X_0)\\
&\hspace*{0.2em} =  \int H_n'(X_s)\d X_s + \frac{1}{2}\int H_n''(X_s)\d s\\
&\hspace*{0.2em} =  \int_0^t\h_n(X_s)b(X_s)\d s + \int_0^t\h_n(X_s)\d W_s-\int_0^t\h_n(X_s)b(X_s)\d s\\
&\hspace*{5em}+ \int_0^t\left\{b(X_s)(1-\phi_n(X_s)) + \frac{\rho_b(x)}{2\rho_b(X_s)}(\phi'_n(X_s)-\rho_b(X_s)) -\frac{\phi'_n(X_s)}{2}\right\}\d s\\
&\hspace*{0.2em} =  \int_0^t\h_n(X_s)\d W_s + \int_0^t\bigg\{b(X_s)(1-\phi_n(X_s)) + \frac{\rho_b(x)}{2\rho_b(X_s)}(\phi'_n(X_s)-\rho_b(X_s)) -\frac{\phi'_n(X_s)}{2}\bigg\}\d s.
\end{align*}
Continuity of diffusion local time $(\L_t^a)_{a\in\R}$ and \eqref{mollifier} imply that
\begin{eqnarray*}
\int_0^t\left\{\frac{\rho_b(x)}{2\rho_b(X_s)}\phi'_n(X_s) -\frac{1}{2}\phi'_n(X_s)\right\}\d s 
&=&\int \left( \frac{\rho_b(x)}{2\rho_b(y)}\phi'_n(y) -\frac{1}{2}\phi'_n(y)\right)\L_t^y(\X)\d y \\
&\rightarrow_{n\to\infty}& \left(\frac{\rho_b(x)}{2\rho_b(x)}-\frac{1}{2}\right)\L_t^x(\X) \ =\ 0.
\end{eqnarray*}
Using the at-most-linear-growth condition on $b$, it can be shown that, for fixed $x\in\mathbb Q$, there exist constants $\theta_1$,$\theta_2>0$ such that, 
for all $n\in\N$,
\[\theta_2 F_b(u)\ \geq\ \phi_n(u), \quad \forall\,u\leq -\theta_1.\]
Intuitively speaking, this relation reflects the fact that $\rho_b$ has tails at least as heavy as a normal distribution.
This implies, for all $n\in\N$,
\[\|\h_n\|_\infty\ \leq\ \frac{2\rho_b(x)}{\inf_{|u|\leq\theta_1}\rho_b(u)}  + \sup_{u\geq 0}\frac{1-F_b(u)}{\rho_b(u)}\rho_b(x) + (\theta_2+1)\sup_{u\leq 0}\frac{F_b(u)}{\rho_b(u)} + 3.\]
Thus, taking account of $\lim_{n\to\infty}\phi_n(u)=\mathds{1}\{u>x\}$ and $\lim_{n\to\infty}\h_n(u)=\h^{f_x}(u)$, we obtain from the dominated convergence theorem and its version for stochastic integrals (see, e.g., Proposition 5.8 in \cite{legall16}) almost surely
\[\int_{X_0}^{X_t}\h^{f_x}(u)\d u\ =\ \int_0^t\h^{f_x}(X_s) \d W_s + \int_0^t b(X_s)\mathds{1}\{X_s\leq x\}\d s  - \frac{t}{2}\rho_b(x).\]
Thus, the martingale approximation from Proposition \ref{prop:mom} and, consequently, Proposition \ref{theo:mart} is valid for the class $\FF_{K,h}$ introduced in \eqref{def:fkh}. 
In particular, there exist positive constants $\co$ and $\Lambda$ such that
\[\sup_{b\in\Sigma} \P_b\left(\left\| \sqrt t\left(\frac{1}{t}\int_0^t f(X_s)\d X_s - \E_b\left[f(X_0)b(X_0)\right]\right)\right\|_{\mathcal{F}_{K,h}}\geq\phi(u) \right)\ \leq\ \e^{-u}\quad\forall\,u\geq 1,\]
where 
\[
\phi(u)\ =\ \sqrt h \VV\e\mathbb L\left\{1+ \log\left(\frac{1}{\sqrt h \VV}\right)+\log\left(t\right) +  u\right\}+ \e\co\frac{u}{\sqrt t}+\e\co\sqrt t \exp\left(-\frac{\Lambda t}{2\e\mom}\right),
\]
with $\VV:=\int |K(z)|\sqrt{|z|}\d z$.
Furthermore, for any $x\in\mathbb Q$,
\begin{eqnarray*}
&&\left|\E_b\left[\left(K_h\ast \mathds{1}_{(-\infty,x]}-\mathds{1}_{(-\infty,x]}\right)(X_0)b(X_0)\right]\right|\\
&&\hspace*{3em}=\ \frac{1}{2}\left|\int \int K_h(z)\left(\mathds{1}_{(-\infty,x]}(z+y) - \mathds{1}_{(-\infty,x]}(y)\right)\rho_b'(y)\d z\d y\right|\\
&&\hspace*{3em}=\ \frac{1}{2}\left|\int \int  K(v)\left(\mathds{1}_{(-\infty,x-vh]}(y) - \mathds{1}_{(-\infty,x]}(y)\right)\rho_b'(y)\d v\d y\right|\\
&&\hspace*{3em}=\ \frac{1}{2}\left|-\int_0^{\infty}K(v) \int_{x-vh}^x\rho_b'(y)\d y \d v + \int_{-\infty}^0 K(v)\int_x^{x-vh}\rho_b'(y)\d y\d v\right|\\
&&\hspace*{3em}=\ \frac{1}{2}\left|\int K(v) (\rho_b(x-vh)-\rho_b(x)) \d v\right|.
\end{eqnarray*}
In case $\beta>1$, we proceed with
\begin{eqnarray*}
&&\frac{1}{2}\left|\int K(v) (\rho_b(x-vh)-\rho_b(x)) \d v\right|\\
&&\hspace*{3em}=\, \frac{1}{2}\left|\int K(v) \sum_{i=1}^{\lfloor\beta\rfloor-1}\frac{\rho_b^{(i)}(x)}{i!}(vh)^{i} + \frac{\rho_b^{\lfloor \beta  \rfloor}(x-\tau_vvh)}{\lfloor \beta  \rfloor!}(vh)^{\lfloor \beta  \rfloor} \d v\right|\\
&&\hspace*{3em}=\, \frac{1}{2}\left|\int K(v)\frac{\rho_b^{\lfloor \beta  \rfloor}(x-\tau_vvh)-\rho_b^{\lfloor \beta  \rfloor}(x)}{\lfloor \beta  \rfloor!}(vh)^{\lfloor \beta  \rfloor} \d v \right|\\
&&\hspace*{3em}\leq\, \frac{\mathcal L}{2\lfloor \beta\rfloor!}\int\left|K(v)\right|\left| (\tau_v vh)^{\beta - \lfloor \beta\rfloor}(vh)^{\lfloor \beta  \rfloor}\right| \d v,
\end{eqnarray*}
where $\tau_v\in [0,1]$, $v\in\R$.
For $\beta\leq 1$, $\rho_b$ is H\"older continuous to the exponent $\beta$ which implies
\begin{eqnarray*}
\frac{1}{2}\left|\int K(v) (\rho_b(x-vh)-\rho_b(x)) \d v\right|&\leq& \frac{\mathcal L}{2}\int\left|K(v)\right|\left|vh\right|^{\beta} \d v.
\end{eqnarray*}
Thus, for $b\in\Sigma(\beta,\ML)$ with $\beta>0$,
\[\left|\E_b\left[\left(K_h\ast \mathds{1}_{(-\infty,x]}-\mathds{1}_{(-\infty,x]}\right)(X_0)\ b(X_0)\right]\right| \ \leq\ h^{\beta}\frac{\mathcal L}{2\lfloor \beta\rfloor!}\int |K(v)v^{\beta }|\d v.\]
In view of \eqref{eq:bx} and the above considerations, for any $u\geq 1$,
\begin{eqnarray}\nonumber
&&\sup_{b\in\Sigma}\P_b\left(\sqrt t\sup_{x\in \mathbb{Q}}|B_x(t,h)|\ \geq\ 2\left(\phi(u)+\sqrt t h^{\beta}\frac{\mathcal L}{2\lfloor \beta\rfloor!}\int |K(v)v^{\beta }|\d v\right)\right)\\\nonumber
&&\hspace*{2em}=\ \sup_{b\in\Sigma} \P_b\left(\sup_{f\in\FF_{K,h}}\Big|\frac{1}{\sqrt t}\int_0^t f(X_s)\d X_s\Big|\ \geq\ \phi(u)+\sqrt t h^{\beta}\frac{\mathcal L}{2\lfloor \beta\rfloor!}\int |K(v)v^{\beta }|\d v\right)\\\nonumber
&&\hspace*{2em}
\leq\ \sup_{b\in\Sigma}\P_b\Bigg(\sup_{f\in\FF_{K,h}}\Big|\frac{1}{\sqrt t}\int_0^t f(X_s)\d X_s-\sqrt t\E_b[f(X_0)b(X_0)]\Big|\\\nonumber
&&\hspace*{8em} + \sqrt t h^{\beta}\frac{\mathcal L}{2\lfloor \beta\rfloor!}\int |K(v)v^{\beta }|\d v\ \geq\ \phi(u)+\sqrt t h^{\beta}\frac{\mathcal L}{2\lfloor \beta\rfloor!}\int |K(v)v^{\beta }|\d v\Bigg)\\\label{eq:B}
&&\hspace*{2em}\leq \ \e^{-u}.
\end{eqnarray}
Set
\[\lambda_0:=\sqrt h \VV\e\mathbb L\left\{1+ \log\left(\frac{1}{\sqrt h\VV}\right)+\log\left(t\right)\right\}+\e\co\sqrt t \e^{-\frac{\Lambda t}{2\e\mom}}+\sqrt t h^{\beta}\frac{\mathcal L}{2\lfloor \beta\rfloor!}\int |K(v)v^{\beta }|\d v.
\]
Define $\Lambda_1:=(8\VV\e\co+ 8\e\co)^{-1}$, and choose $\Lambda_0\geq 1$ such that, for all $t\geq 1$, $h\in(0,1)$,
\[
 8h t^{-1/2}\int_\R |K(z)z|\d z <  4\Lambda_0\lambda_0
\]
and $\VV\e\co\Lambda_1\Lambda_0>1$.
Taking into account \eqref{eq:A}, this choice in particular implies that, for any $\lambda\geq 8\Lambda_0\lambda_0$, 
\begin{align}\label{eq:C}\begin{split}
&\sup_{b\in\Sigma}\,\,\P_b\left(\sqrt t\|\rho_{t,K}(h)-\rho_t^\circ\|_\infty>\lambda\right)\\
&\qquad\leq\ \sup_{b\in\Sigma}\,\,\P_b\left(\sqrt t\sup_{x\in\mathbb{Q}}|A_{1,x}(t,h)+A_{2,x}(t,h)+B_x(t,h)|>\lambda\right)\\
&\qquad\leq\ \sup_{b\in\Sigma}\,\, \P_b\left(\sqrt t\|B_\bullet(t,h)\|_\infty>\lambda-\frac{8h}{\sqrt t}\int|K(z)z|\d z\right)\\
&\qquad\leq\ \sup_{b\in\Sigma}\,\, \P_b\left(\sqrt t\|B_\bullet(t,h)\|_\infty>\lambda-4\Lambda_0\lambda_0\right)\ 
\leq\ \sup_{b\in\Sigma}\,\, \P_b\left(\sqrt t\|B_\bullet(t,h)\|_\infty>\lambda/2\right).
\end{split}
\end{align}
 Note that, for $u=\Lambda_1\lambda h^{-1/2}$,
\begin{eqnarray*}
\phi(u) + \sqrt t h^{\beta}\frac{\mathcal L}{2\lfloor \beta\rfloor!}\int |K(v)v^{\beta }|\d v&\leq& \lambda_0 + \sqrt h\VV\e\co u + \frac{\e\co u}{\sqrt t}\\
&\leq& \lambda_0\Lambda_0 + \frac{u}{8}\Lambda_1^{-1}\sqrt h\ \leq\ \lambda_0\Lambda_0+\frac{\lambda}{8}\ \leq\ \frac{\lambda}{4}.
\end{eqnarray*}
Summarising, \eqref{eq:B} and \eqref{eq:C} then give the asserted inequality \eqref{eq:theo:1}.
\end{proof}

\medskip

We are now in a position to derive the announced upper bounds on the moments of centered diffusion local time.

\begin{proof}[Proof of Corollary \ref{cor:centlt}]
We point out that the assumption that $b\in\Sigma(\C,A,\gamma,1)$ already imposes some regularity on the invariant density in the sense of Definition \ref{def:Hölder}. 
More precisely, if $b\in\Sigma(\C,A,\gamma,1)$, the invariant density $\rho_b$ is bounded and Lipschitz continuous due to \eqref{eq:reg_rho_b} which in turn means that $b\in \Sigma(1,\ML)$.
Decompose
\begin{equation}\label{ineq:decomp}
\left(\E_b\left[\bigg\|\frac{\L_t^\bullet(\X)}{t}-\rho_b\bigg\|_\infty^p\right]\right)\p\ \leq\ \left(\E_b\left[\left\|\rho_t^\circ-\rho_{t,K}\left(t^{-1}\right)\right\|_\infty^p\right]\right)\p+\left(\E_b\left[\|\rho_{t,K}\left(t^{-1}\right)-\rho_b\|_\infty^p\right]\right)\p.
\end{equation}
Inspection of the proof of Theorem \ref{theo:cath} shows that, for any $b\in\Sigma(1,\mathcal L)$ and $h=h_t\geq t^{-1}$,
\begin{eqnarray*}
\left(\E_b\left[\|\rho_{t,K}(h)-\rho_t^\circ\|_\infty^p\right]\right)\p
&\leq& \left(\E_b\left[\|A_{1,x}(t,h) + A_{2,x}(t,h) + B_x(t,h)\|_\infty^p\right]\right)\p\\
&\leq&  \frac{8h}{t}\int|K(z)z|\d z + \e^{-1}\varphi(p) + \frac{h\mathcal L}{2}\int |K(v)v|\d v, 
\end{eqnarray*}
where 
\[
\varphi(u)\ =\ \VV\e\mathbb L\sqrt{\frac{h}{t}} \left\{1+ \log\left(\frac{1}{\sqrt h \VV}\right)+\log\left(t\right) +  u\right\}+ \e\co\frac{u}{t}+\e\co \exp\left(-\frac{\Lambda t}{2\e\mom}\right).\]
The second term on the rhs of \eqref{ineq:decomp} is bounded by means of Corollary \ref{prop:csi}.
Consequently, specifying $h=h_t\sim t^{-1}$, we obtain a constant $\zeta$ such that
\begin{eqnarray*}
\left(\E_b\left[\bigg\|\frac{\L_t^\bullet(\X)}{t}-\rho_b\bigg\|_\infty^p\right]\right)\p
&\leq& \zeta\left(\frac{1}{\sqrt t}\left\{1+\sqrt{\log(pt)} + \sqrt p\right\} + \frac{p}{t}+t\exp\left(-\frac{\Lambda t}{2\e\mom}\right)\right).
\end{eqnarray*}
\end{proof}

\medskip

\begin{proof}[Proof of Theorem \ref{improved_version}]
Analogously to the proof of Proposition \ref{theo:mart}, we start with decomposing $\H_t$ into finite variation and martingale part,
\begin{eqnarray*}
\H_t(f)&=& \sqrt t\left(\frac{1}{t}\int_0^t f(X_s)b(X_s)\d s - \int(fb)\d\mu_b + \frac{1}{t}\int_0^t f(X_s)\d W_s\right)\\
&=& \G^b_t(f) + \frac{1}{\sqrt t}\int_0^t f(X_s)\d W_s.
\end{eqnarray*}
For the finite variation part, part $(\mathbf{II})$ of Theorem \ref{theo:bernsup} gives, for any $p\geq 1$,
\begin{eqnarray*}
\left(\E\left[\|\G^{b}_t\|_{\mathcal{F}}^p\right]\right)^{\p}
\ \leq\ \Ph_t(p),
\end{eqnarray*}
for $\Ph_t$ defined as in \eqref{def:Ph}.
It remains to bound the $p$-th moments of the original stochastic integral term. Given $f\in\FF\cup \overline\FF$ and any $p\geq 2$, it holds
\begin{eqnarray*}
\left(\E_b\left[\left|\frac{1}{\sqrt t}\int_0^tf(X_s)\d W_s\right|^p\right]\right)^{\frac{1}{p}}
&\leq& \bdg \sqrt p\left(\E_b\left[\left(\frac{1}{t}\int_0^tf^2(X_s)\d s\right)^{\frac{p}{2}}\right]\right)^{\frac{1}{p}}\\
&\leq& \bdg \sqrt p\|f\|_{L^2(\lebesgue)}\left(\E_b\left[\left(\frac{1}{t}\|\L_t^\bullet(\X)\|_\infty\right)^{\frac{p}{2}}\right]\right)^{\frac{1}{p}}.
\end{eqnarray*}
In the same way, we obtain for $1\leq p<2$, 
\begin{eqnarray*}
\left(\E_b\left[\left|\frac{1}{\sqrt t}\int_0^tf(X_s)\d W_s\right|^p\right]\right)^{\frac{1}{p}}
&\leq& \frac{1}{\sqrt t}\ \left(\E_b\left[\left(\int_0^tf(X_s)\d W_s\right)^{2p}\right]\right)^{\frac{1}{2p}}\\
&\leq &\overline\bdg \sqrt{\frac{2p}{t}}\left(\E_b\left[\left(\int_0^tf^2(X_s)\d s\right)^p\right]\right)^{\frac{1}{2p}}\\ 
&=& \bdg\sqrt{\frac{p}{t}}\left(\E_b\left[\left(\int_\R f^2(y)\L_t^y(\X)\d y\right)^p\right]\right)^{\frac{1}{2p}}\\
&\leq&\bdg \sqrt p\|f\|_{L^2(\lebesgue)}\left(\E_b\left[\left(\frac{1}{t}\|\L_t^\bullet(\X)\|_\infty\right)^{p}\right]\right)^{\frac{1}{2p}}.
\end{eqnarray*}
It follows from Corollary \ref{cor:centlt} that there exists positive constants $\bar{\co}_1$, $\tilde{\co}_1$ such that, for any $p\geq 1$ and $t\geq 1$,
\begin{eqnarray*}
&&\max\left\{\left(\E_b\left[\left(\frac{1}{t}\sup_{a\in\mathbb Q}\left|L_t^a(X)\right|\right)^{\frac{p}{2}}\right]\right)^{\frac{1}{p}},\ 
\left(\E_b\left[\left(\frac{1}{t}\sup_{a\in\mathbb Q}\left|L_t^a(X)\right|\right)^{p}\right]\right)^{\frac{1}{2p}}\right\}\\
&&\hspace*{7em} \leq\  \bar{\co}_1\Bigg(1+t\exp\left(-\frac{\Lambda t}{2\e\mom}\right) + \frac{1}{\sqrt t}\left\{1+\sqrt{\log t}+ \sqrt p\right\} + \frac{p}{ t}  \Bigg)^{1/2}\\
&&\hspace*{7em} \leq\ \tilde{\co}_1\Bigg(1+\left(\frac{p}{t}\right)^{1/4} + \sqrt{\frac{p}{t}}\Bigg).
\end{eqnarray*}
Consequently, for any $p\geq 1$,
\begin{eqnarray*}
2\sup_{f\in\FF}\left(\E_b\left[\Big|\frac{1}{\sqrt t}\int_0^t f(X_s)\d W_s\Big|^p\right]\right)\p&\leq&\Lambda_3\V \left(\sqrt p+\frac{p}{t^{1/4}}\right),\\
\left(\E_b\left[\Big|\frac{1}{\sqrt t}\int_0^t (f-g)(X_s)\d W_s\Big|^p\right]\right)\p&\leq&\Lambda_3\|f-g\|_{L^2(\lebesgue)}\left(\sqrt p+\frac{p}{t^{1/4}}\right),\quad f,g\in \FF,
\end{eqnarray*}
with $\Lambda_3:=\max\{4\tilde{\co}_1\bdg,1\}$.
In view of Lemma \ref{moment_tail_lemma}, this last estimate implies that, for any $u\geq 1$, $f,g\in\mathcal F$,
\[
\P_b\left(\Big|t^{-1/2}\int_0^t(f-g)(X_s)\d W_s\Big|\ \geq\ d(f,g) \left(\sqrt u+t^{-1/4}u\right) \right)\ \leq\ \exp(-u),
\]
for $d(f,g):=\e\Lambda_3\ \|f-g\|_{L^2(\lebesgue)}$.
Analogously to the proof of Proposition \ref{theo:mart}, we obtain for all $p,q\geq 1$, $k\in\N_0$,
\begin{align*}
\frac{1}{\sqrt t}\left\|\sup_{f\in\FF_k}\int_0^t f(X_s)\d W_s\right\|_{q}&\leq\
 \frac{\tilde C_1}{t^{1/4}}\int_0^\infty \log N(u,\FF_k,d)\d u + \tilde C_2\int_0^\infty\sqrt{\log N(u,\FF_k,d)}\d u\\
&\hspace*{14em} +  2\sup_{f\in\FF}\frac{1}{\sqrt t}\left\|\int_0^t f(X_s)\d W_s\right\|_q,
\end{align*}
and from the local result, we infer, for any $p\geq1$,
\begin{align*}
\frac{1}{\sqrt t}\left(\E_b\left[\bigg\|\int_0^tf(X_s)\d W_s\bigg\|_{\FF}^p\right]\right)^{\frac{1}{p}}&\leq\
 \frac{\tilde C_1}{t^{1/4}}\sum_{k=0}^\infty E(\FF_k,d,1)\e^{-k/2}+ \tilde C_2\sum_{k=0}^\infty E(\FF_k,d,2)\e^{-k/2}\\
&\hspace*{14em} +  6\Lambda_3 \V \left(\sqrt p+\frac{p}{t^{1/4}}\right).
\end{align*}
The upper bounds for the entropy integrals from Lemma \ref{EntropyIntegralLebesgueMetric} finally imply that, for any $p\geq 1$,
\[
\left(\E_b\|\H_t\|_\FF^p\right)\p\ \leq\ \Ph_t(p) + \Ps_t(p),\]
with
\begin{align*}
\Ps_t(p)&:=\ 6\V\e\Lambda_3\left\{\frac{\tilde C_1 v}{t^{1/4}}\left(1+ \log\left(\frac{\AA}{\V}\sqrt{\mathcal S + p\Lambda t}\right)\right)+
2\tilde C_2\sqrt{v\log\left(\frac{\AA}{\V}\sqrt{\mathcal S + p\Lambda t}\right)}\right\}\\
&\hspace*{24em} + 6\Lambda_3\V \left(\sqrt p+\frac{p}{t^{1/4}}\right).
\end{align*}
\end{proof}

\end{appendix}

\bibliography{condiff} 

\end{document}